\documentclass[11pt, english]{amsart}
\usepackage{amssymb,amsthm,amsmath,amstext}
\usepackage{mathrsfs}
\usepackage{bm}
\usepackage[utf8]{inputenc}
\usepackage{amscd}
\usepackage{tikz}
\usepackage{version}\usepackage{marginnote}
\usepackage{comment}

\usepackage[color,matrix,arrow]{xy}
\xyoption{all}
\usepackage{color}
\usepackage{pgfplots}

\DeclareMathOperator{\Coh}{\textbf{Coh}}

\DeclareMathOperator{\NS}{NS}
\DeclareMathOperator{\Stab}{Stab}

\makeatletter
\ifnum\@ptsize=0  \fi \ifnum\@ptsize=2  \fi 

\title{A birational involution}
\author{Pietro Beri, Laurent Manivel}
\date{}

\address{Institut Elie Cartan de Lorraine, Université de Lorraine et CNRS, 54000 Nancy, France}
\email{pietro.beri@univ-lorraine.fr}

\address{Institut de Math\'ematiques de Toulouse ; UMR 5219, Universit\'e de Toulouse \& CNRS, F-31062 Toulouse Cedex 9, France}
\email{manivel@math.cnrs.fr}

\color{black}

\begin{document}

\theoremstyle{plain}
\newtheorem{theorem}{Theorem}
\newtheorem*{conjecture*}{Conjecture}
\newtheorem{prop}[theorem]{Proposition}
\newtheorem{lemma}[theorem]{Lemma}
\newtheorem{coro}[theorem]{Corollary}

\def\AA{{\mathbb{A}}}
\def\CC{{\mathbb{C}}}
\def\RR{{\mathbb{R}}}
\def\OO{{\mathbb{O}}}
\def\HH{{\mathbb{H}}}
\def\PP{{\mathbb{P}}}
\def\QQ{{\mathbb{Q}}}
\def\ZZ{{\mathbb{Z}}}
\def\SS{{\mathbb{S}}}
\def\cO{{\mathcal{O}}}
\def\cC{{\mathcal{C}}}\def\cA{{\mathcal{A}}}
\def\cD{{\mathcal{D}}}
\def\cE{{\mathcal{E}}}
\def\cF{{\mathcal{F}}}
\def\cG{{\mathcal{G}}}
\def\cH{{\mathcal{H}}}
\def\cI{{\mathcal{I}}}
\def\cJ{{\mathcal{J}}}
\def\cK{{\mathcal{K}}}
\def\cL{{\mathcal{L}}}
\def\cP{{\mathcal{P}}}
\def\cQ{{\mathcal{Q}}}
\def\cR{{\mathcal{R}}}
\def\cT{{\mathcal{T}}}
\def\cU{{\mathcal{U}}}
\def\cV{{\mathcal{V}}}\def\cW{{\mathcal{W}}}
\def\cZ{{\mathcal{Z}}}\def\cX{{\mathcal{X}}}
\def\ra{{\rightarrow}}
\def\lra{{\longrightarrow}}
\def\ft{{\mathfrak t}}\def\fs{{\mathfrak s}}\def\fn{{\mathfrak n}}
\def\fg{{\mathfrak g}}\def\fp{{\mathfrak p}}\def\fb{{\mathfrak b}}
\def\fso{\mathfrak{so}}
\def\fsp{\mathfrak{sp}}
\def\fsl{\mathfrak{sl}}\def\fgl{\mathfrak{gl}}

\let\Iff\iff
\def\iff{\Leftrightarrow}

\begin{abstract}
    Given a \textcolor{black}{very} general K3 surface $S$ of degree $18$, lattice theoretic considerations allow to predict the existence of an anti-symplectic birational involution $\varphi$ of the Hilbert cube $S^{[3]}$. We describe this 
    involution in terms of the Mukai model of $S$,  \textcolor{black}{with the help of the geometry of the exceptional Lie group $G_2$}. We make a connection with Homological Projective Duality by showing that the indeterminacy locus of the involution is  a $\PP^2$-fibration over the dual K3 surface of degree two. 
    We deduce that $\varphi$ is an instance of a Mukai flop. 
\end{abstract}

\maketitle

\tableofcontents 

\section{Introduction}

Consider 
a complex algebraic $K3$ surface $S$ whose Picard group is generated by an ample line bundle $\mathcal{H}$ such that $\mathcal{H}^2=2t$ 
(a very general element of the $19$-dimensional moduli space of $2t$-polarized $K3$ surfaces). 
A classification of the group of biregular automorphisms of the punctual Hilbert scheme $S^{[n]}:= Hilb^n(S)$ has been given by Boissière, An. Cattaneo, Nieper-Wi\ss kirchen and Sarti \cite{bcns} for $n= 2$, and by Al. Cattaneo for all $n\ge 3$ \cite{cattaneo}. In particular, $Aut(S^{[n]})$ is either trivial or generated by an involution which is non-symplectic for $t\geq 2$.
More recently, extended lattice theoretic considerations allowed Al. Cattaneo and the first author to decide whether $S^{[n]}$ admits non-trivial birational automorphisms \cite[Theorem 1.1]{beri}, symplectic or not.  When it does, there exists a unique 
such automorphism, a birational involution that may be symplectic or non symplectic. The precise result is the following:

\begin{theorem}\label{abstract classification}
Suppose $t\ge 2$. 
Then $Bir(S^{[n]})$ is either trivial or generated by a birational involution $\varphi$. The latter holds 
if and only if:
\begin{enumerate}
    \item $d=t(n-1)$ is not a square,
    \item the minimal non-trivial solution $(X,Y)$ of Pell’s equation $$ X^2-dY^2= 1 \quad with\quad  X=\pm 1 \; mod \;n-1$$
    is such that 
 $Y$ is even and  $(X,X)=(1,1),(1,-1)$ or $(-1,-1)$ in $\ZZ/2(n-1)\ZZ\times\ZZ/2t\ZZ$.
\end{enumerate} 
\end{theorem}

When $(X,X)=(1,1)$, the involution is symplectic and never biregular. When $(X,X)=(\pm 1,-1)$, it is always non-symplectic, and biregular if and only if it fixes an ample class of $S^{[n]}$; for that case  the classification can be found in \cite{bcns,cattaneo}, the precise condition for regularity being condition (iii) of \cite[Theorem 1.3]{cattaneo}.

Such a statement does not provide much insight on how to construct this birational involution geometrically, when it does exist. 
Some cases have been explicitly described in \cite[Examples 6.1, 6.2]{beri}, for $t=2$ and $n=6, 8, 18$. 
A few other cases had been known before:
\begin{itemize}
    \item Suppose $t=n$, so that $\mathcal{H}$ embeds $S$ into $\PP^{n+1}$, as a surface of  degree $2n$. 
    If we choose $n$ points $p_1,\ldots , p_n$ in $S$ in general linear position, 
    they generate a codimension two subspace of $\PP^{n+1}$, which in general cuts $S$ at $2n$ distinct points $p_1,\ldots , p_n$
    and $q_1,\ldots , q_n$. This yields the {\it Beauville involution}, which is non-symplectic, and 
    biregular only for $n=2$ (that is for quartics in $\PP^3$). 
    \item Suppose $n=2$ and $t=5$, so that $S$ is embedded in the Grassmannian $G(2,5)$ as the transverse intersection of 
     a quadric $Q$ and three hyperplanes $H_1, H_2, H_3$. In particular, the linear span of $S$ intersects the Grassmannian along  $F=G(2,5)\cap H_1\cap H_2\cap H_3$, 
     which is a smooth Fano threefold of index two. Now consider two general points $p_1,p_2$ in $S$, defining two transverse
     planes $P_1, P_2$ in $\CC^5$. The Grassmannian $G(2,P_1\oplus P_2)$ is a four dimensional quadric in $G(2,5)$, its intersection with
     $F$ is a conic, and its intersection with $S=F\cap Q$ therefore consists in general of four points $p_1,p_2$ and $q_1,q_2$. 
     This yields the O'Grady involution \cite[section 4.3]{ogrady}.
      \item Suppose $n=3$ and again $t=5$,
 so that $S$ can be described as before. Given three general points $p_1,p_2,p_3$ in $S$, defining three 
     planes $P_1, P_2, P_3$ in $\CC^5$, or equivalently three lines $\ell_1, \ell_2, \ell_3$ in $\PP^4$,
 it is a classical fact that there exists in general a unique line  $\ell_0$ meeting these three lines. 
 Equivalently, if $\Sigma_\ell\subset G(2,5)$
 is the codimension two Schubert cycle parametrizing lines in $\PP^4$ that meet a given line $\ell$, then $\Sigma_{\ell_1}.\Sigma_{\ell_2}.\Sigma_{\ell_3}=1$. Since $\Sigma_\ell$ has degree three, the intersection $\Sigma_{\ell_0}\cap S$
 consists of the three points $p_1,p_2,p_3$ (since $\ell_0\in \Sigma_{\ell_i}$ is equivalent to  $\ell_i\in \Sigma_{\ell_0}$), and 
 three other points $q_1,q_2,q_3$ \cite[Example 4.12]{debarre}.
  \end{itemize}
  
  \medskip
In fact, there are some easily identified infinite sequences of pairs $(n,t)$, described in \cite[Prop. 2.6 (i)]{beri}, that  satisfy the conditions of Theorem $1$: fix any $k>0$ and take $t=(n-1)k^2+1$,  for 
any  $n\geq 2$. For $k=1$ we recover the Beauville involution. For $k=2$ only the case $n=2$ has
been described, this is the O'Grady involution. In this paper, we describe the next
case,  $n=3$ and $t=9$. 

\smallskip
A general $K3$ surface of degree $2t=18$, or equivalently of genus $10$, admits a Mukai model: it can be described as a codimension three 
linear section of 
the adjoint variety of the exceptional Lie group $G_2$, the closed $G_2$-orbit inside the projectivized adjoint representation
\cite{mukai}. 
We denote this five dimensional homogeneous space by $X_{ad}(G_2)\subset \PP(\fg_2)$, so that 
$$S=X_{ad}(G_2)\cap L$$
for some codimension three linear subspace $L$ of $\PP(\fg_2)$.

Recall  that $G_2$ can be described as a subgroup of $SO(7)$  \cite{baez}. This means in particular that $G_2$ has a natural representation $V_7$
of dimension $7$, which is irreducible, and that it preserves some non-degenerate quadratic form $Q$. Since then $\fg_2\subset \fso_7\simeq \wedge^2V_7$, we deduce that the adjoint variety $X_{ad}(G_2)$, which is the minimal orbit in $ \PP(\fg_2)$, must be a subvariety
of $G(2,V_7)$. There is a diagram
$$\begin{array}{ccccc}
  X_{ad}(G_2) & \hookrightarrow & OG(2,V_7) & \hookrightarrow &  G(2,V_7) \\
  \downarrow  & &\downarrow  & & \downarrow \\
  \PP(\fg_2) & \hookrightarrow  & \PP(\wedge^2V_7)& =  & \PP(\wedge^2V_7)
\end{array}$$
where the vertical arrows are embeddings. In fact, the Mukai model for $S$ is constructed from the unique stable vector bundle $\cU_2^\vee$ on $S$ with Mukai vector $(2,\cH,5)$, where $\cH^2=18$: this bundle is generated by its seven global sections, which embed $S$ into $G(2,V_7)$.

\medskip
So let $p_1,p_2,p_3$ be generic points on $S$. They can be identified with three planes $P_1,P_2,P_3\subset V_7$. In general, these
three planes generate a hyperplane $V_6\subset V_7$, \textcolor{black}{and $F\ell(V_6):=X_{ad}(G_2)\cap G(2,V_6)$  is a copy of the complete flag variety $F\ell_3$ of $\fsl_3$. }  
Since $F\ell_3$ has dimension three and degree $6$, its intersection with $L$  consists in general 
of six distinct points,  $p_1,p_2,p_3$.
plus three other points $q_1,q_2,q_3$.  
We define
\begin{equation}\label{phi}
\begin{array}{l}
    \varphi:S^{[3]}  \dashrightarrow  S^{[3]},\qquad 
    p_1+p_2+p_3      \mapsto    q_1+q_2+q_3.
\end{array}
\end{equation}
Note that,  at this stage, we can consider any K3 surface $S$ of genus ten whose projective model is a three dimensional section of $X_{ad}(G_2)$. 

\begin{theorem}\textcolor{black}{\textit{(Theorem \ref{involution})}}\label{invo}
For any 
K3 surface $S=X_{ad}(G_2)\cap L$, $\varphi$ is a non-trivial birational involution of $S^{[3]}$.
\end{theorem}

We also provide another interpretation of $\varphi$ in terms of the extremal ray $\mathcal{H}_3-2\delta$ 
of $S^{[3]}$, where $\mathcal{H}_3$ is the line bundle induced by the polarization $\mathcal{H}$ of $S$ and $2\delta$ is,
as usual, the class of the locus of non-reduced subschemes. 

\begin{theorem}\textcolor{black}{(Theorem \ref{basepointfree}, Proposition \ref{degree fiber}, Corollary \ref{phi is covering inv})}\label{bpf}
Let $S$ be a very general K3 surface of genus ten.
The linear system  $|\mathcal{H}_3-2\delta|$ on the punctual Hilbert scheme $S^{[3]}$ is base point free.
The associated morphism $\phi_{|\mathcal{H}_3-2\delta|}$ is generically finite of degree two, and the birational involution $\varphi$ is the corresponding deck transformation.
\end{theorem}

Moreover the indeterminacy locus $I\subseteq S^{[3]}$ of $\varphi$ can be described in terms of 
subschemes of length three that do not span
a hyperplane in $V_7$. Such subschemes are obtained from certain cubic scrolls in the adjoint variety, parametrized 
by what we call decomposing five-planes: those codimension two subspaces of $V_7$ over which the $G_2$-invariant 
three-form is completely decomposed. \textcolor{black}{Part of our arguments rely on a good understanding of these codimension two subspaces.} We show:

\begin{prop}\textit{(Corollary \ref{iota})}\label{projeq}
If $\omega\in \bigwedge^3 V^\vee_7$ is a $G_2$-invariant 3-form, then the subvariety of $G(5, V_7)$ 
parameterizing $\omega$-decomposing five-planes is projectively equivalent to the orthogonal Grassmannian $OG(2,7)$.
\end{prop}

This statement is closely related to a recent paper by Guseva \cite{guseva}.

\medskip
Cubic scrolls in $X_{ad}(G_2)$ and its hyperplane sections were  studied by M. Kapustka and Ranestad \cite{kr}; they proved the remarkable fact that a smooth 
hyperplane section $F=X_{ad}(G_2)\cap H$ is  isomorphic to a linear section of the Grassmannian $G(3,U_6)$, and deduced that cubic scrolls in $F$ are parametrized by the disjoint union of two Veronese surfaces. 
We integrate their description with results from \cite{kuznetsov-hpd}, thus extending it to the case where $F$ is a nodal hyperplane section of $X_{ad}(G_2)$.

When $H$ contains $L$, the section $F$ contains the K3 surface $S$ and we can pull-back the tautological and quotient bundles on $G(3,U_6)$ to get two vector bundles on $S$; but only one when we pull-back from $LG(3,U_6)$, where the  tautological and quotient bundles gets identified. This realizes the moduli space $M_\cH(3,\cH, 3)$ of $\cH$-stable vector bundles on $S$ with the corresponding invariants, as a smooth double cover of $L^\perp\simeq\PP^2$,  branched over its intersection with the dual hypersurface to $X_{ad}(G_2)$, which is  a sextic hypersurface. This double cover, described in \cite{kr}, is precisely the degree two K3 surface $\Sigma$, derived equivalent to $S$ up to the twist by a Brauer class, that appears in Kuznetsov's Homological Projective Duality for  $X_{ad}(G_2)$ 
 \cite[Corollary 9.10]{kuznetsov-hpd}. Note that although the moduli space
 $\Sigma=M_\cH(3,\cH, 3)$ is not fine, there exists a  twisted universal bundle 
$\cF$ on $S\times\Sigma$.
 With this in mind, we can describe the indeterminacy locus:

\begin{theorem}\textit{(Proposition \ref{P=I})}
\label{structure}
Let $S$ be a very general K3 surface of genus ten.
The indeterminacy locus $I$ of $\varphi$ on the punctual Hilbert scheme $S^{[3]}$ is a smooth fourfold with an étale locally trivial $\PP^2$-fibration over the smooth degree two $K3$ surface $\Sigma$, given either by 
\begin{enumerate}
\item the locus of length-three subschemes of $S$ contained in $G(2,V_5)$ for some decomposing five-plane $V_5\subset V_7$;
    \item the family of cubic scrolls on hyperplane sections of $X_{ad}(G_2)$ containing $S$;
    \item the projectivization $\PP_\Sigma(\cG)$ of the twisted vector bundle $\cG=p_{2*}{\mathcal Hom}(\cF,p_1^*\cU_2)$
over $\Sigma$, where $p_1$  and $p_2$ are the projections from $S\times\Sigma$ to $S$ and $\Sigma$, respectively.
    \end{enumerate}
\end{theorem}

Finally, since $I$ is a smooth codimension two subvariety of $S^{[3]}$, fibered in 
$\PP^2$'s, we can perform a Mukai flop to obtain another smooth, simply connected sixfold $M$ with a non-degenerate two-form; but a priori not necessarily hyperKähler. It turns out that $M$ is in fact isomorphic with $S^{[3]}$ and that, up to identifying the two models, the Mukai flop is nothing else than our involution $\varphi$!

\begin{theorem}\label{intro: flopdeMukai} \textit{(Corollary \ref{flopdeMukai})}
For $S^{[3]}$ the punctual Hilbert scheme over a very general K3 surface of genus ten,
$\varphi$ is a Mukai flop.
\end{theorem} 

The techniques we use are a mixture of $G_2$-geometry with moduli interpretations and stability conditions. 
Basic facts about $G_2$ are recalled in section $2$, where we also discuss decomposing five-planes and cubic scrolls. 
The involution $\varphi$ is defined in section $3$,
where we prove Theorem \ref{invo} and we put into the picture the extremal contraction induced by a multiple of $\cH_3-2\delta$, which ultimately leads to Theorem \ref{intro: flopdeMukai}. These two results can be seen as different descriptions of $\varphi$ and the key point to identify them is Theorem \ref{intro: flopdeMukai}.
The connections with the linear system $|\mathcal{H}_3-2\delta|$ itself are discussed in section $4$: we first prove base point freeness, then that the associated morphism is generically finite of degree two. 

\smallskip All along we will find strong analogies, but also important differences, with the story 
of Gushel-Mukai varieties and (double) EPW sextics, a story which includes the O'Grady involution. Some of those observations are gathered in section \ref{defo}, where we also briefly discuss the indirect and intriguing connections, also present in  \cite{kuznetsov-hpd}, of some of our constructions with the Grothendieck-Springer resolution. Being 
able to upgrade the present work in order to construct a new locally complete family
of polarized hyperK\"ahler manifolds, is one of the challenges that remain to be met. 

\smallskip 
One can also wonder to which extent our constructions can be generalized. The definition of our birational involution still makes sense for certain punctual Hilbert schemes of K3 surfaces with different polarizations, and the geometric study presented in this paper can be partially extended, although we of course lose all the interpretations we gave in terms of $G_2$-geometry \cite{more}. Parallel constructions, presented in the more flexible setting of derived categories of moduli spaces of sheaves on K3's, were very recently presented in \cite{faenzi-menet}. 
But the case of K3 surfaces of genus ten, with its so many facets,  will certainly remain a unique gem.

\medskip\noindent {\it Acknowledgements.} We warmly thank G. Kapustka, E. Macri,  G. Mongardi, K. O'Grady
for their comments and hints, and A. Kuznetsov for his careful reading of a first draft of the manuscript. \textcolor{black}{We also thank an anonymous referee for the lucid comments and insightful
suggestions}, that allowed to improve the paper substantially. In particular the proof of 
Theorem \ref{basepointfree} presented in section 7 is entirely due to her and is completely different from our 
original argument.
We also thank C. Bai, D. Mattei, M. Varesco for useful discussions. 

P. Beri is supported by the ERC Synergy Grant HyperK, agreement ID 854361.
L. Manivel acknowledges support from the ANR project FanoHK, grant
ANR-20-CE40-0023.

\section{Some $G_2$ geometry} 

\subsection{Classical preliminaries}
\textcolor{black}{A useful, wide-ranging reference for the exceptional Lie group $G_2$ is \cite{baez}.}
There are two classical ways to understand $G_2$, which may be both useful according to the context. 
Here we work over the complex numbers, and most of what follows is translated from the real setting.
In particular, each time in the sequel we deal with octonions and the Cayley algebra $\OO$, we will in fact 
be dealing with the {\it complexified} octonions and the {\it complexified} Cayley algebra. We will not
use any special notation for this complexification, hoping this will not cause any confusion. 
\begin{enumerate} 
\item (Cartan 1914) $G_2$ can be defined as the 
automorphism group of the Cayley algebra of octonions. Since the Cayley algebra $\OO$ is normed, it admits a $G_2$-invariant 
quadratic form $Q$, and it splits orthogonally as $\OO=\CC {\bf 1} \oplus V_7$, where $V_7$ is the space of imaginary octonions. 
\item (Engel 1900) $G_2\subset SL(V_7)$ can be defined as the stabilizer of a generic skew-symmetric three-form $\omega\in\wedge^3V_7^\vee$. Then $V_7$ inherits an invariant quadratic form, defined up to scalar by 
$$Q(x,y)=\iota_x\omega\wedge \iota_y\omega\wedge \omega \in \wedge^7V_7^\vee\simeq \CC,$$
where $\iota_x$ denotes the contraction by the vector $x$.
Conversely, on the space of imaginary octonions the skew-symmetric three-form is given by 
the formula $\omega (x,y,z)=\mathrm{Re}(x(yz))$ (where $\mathrm{Re} $ must be interpreted as the projection on the invariant part in the decomposition $\OO=\CC\oplus V_7$). 
From this three-form one easily recovers the octonionic product. 
\end{enumerate}

The two fundamental representations of $G_2$ are $V_7$ and the adjoint representation on the Lie algebra $\fg_2$, whose dimension is $14$. 
The two associated generalized Grassmannians (orbits of highest weight vectors) are $G_2/P_1=\QQ^5\subset \PP(V_7)$, the invariant quadric, and the 
adjoint variety $X_{ad}(G_2)=G_2/P_2\subset\PP(\fg_2)$. There is a fundamental correspondence 
\begin{equation}\label{fundamental}
\xymatrix{
 & G_2/B\ar@{->}[ld]_{p_\QQ}\ar@{->}[rd]^{p_X} & \\
 \QQ^5 & & X_{ad}(G_2),
}
\end{equation}
where the two projections are $\PP^1$-fibrations over Fano fivefolds, $\QQ^5$ of index five and $X_{ad}(G_2)$ of index three; both of them dominated by the flag manifold $G_2/B$, where $B$ is a Borel subgroup containing the maximal parabolic subgroups $P_1$ and $P_2$. 

Since $\fg_2\subset\fso_7 \simeq\wedge^2V_7$, the adjoint variety  $X_{ad}(G_2)\subset \PP(\fg_2)$ is a subvariety of $G(2,V_7)\subset\PP(\wedge^2V_7)$,
hence parametrizes some special planes. According to our two viewpoints:
\begin{enumerate} 
\item $X_{ad}(G_2)$ parametrizes the family of {\it null-planes} in $V_7$, that is, those planes $P\subset V_7=Im(\OO)$
on which the octonionic product vanishes identically \cite{lm-freud, cg}.
\item Equivalently, $X_{ad}(G_2)$ parametrizes the planes $P=\langle x,y\rangle \subset V_7$ such that $\iota_x\iota_y\omega=\omega(x,y,\bullet)=0$. 
Let $\cU_2$ and $\cQ_5$ denote the tautological and quotient vector bundles 
over $G(2,V_7)$; 
this shows that $\omega$
defines a global section of the  bundle $\wedge^2\cU_2^\vee\otimes\cQ_5^\vee = \cQ_5^\vee(1)$,  whose zero-locus is precisely the adjoint variety \cite[Section 8]{kuznetsov-hpd}.
\end{enumerate}
For our purposes, we will mostly adopt this second point of view, but it may be useful to keep the first one in mind as well.

Note that in diagram (\ref{fundamental}), $G_2/B$ is the projectivized bundle $\PP(\cU_2)$ over $X_{ad}(G_2)$. Similarly, it is the projectivization of a $G_2$-equivariant rank two 
bundle $\mathcal{C}$ on $\QQ^5$ called the {\it Cayley bundle} \cite{ottaviani-cayley}, and such that 
$$H^0(\QQ^5,\cC^\vee)=\fg_2^\vee\simeq\fg_2.$$ 

\smallskip
As in \cite{cg} we fix a Cartan subalgebra $\ft\subset\fg_2$, which yields a root space decomposition, and we denote by $\alpha, \beta , \gamma$ three 
short roots summing to zero. Then $\alpha-\beta$, $\beta-\gamma$, $\gamma-\alpha$ are three long 
roots also summing to zero, and we get the twelve roots of $\fg_2$ by including the opposites of those  six roots.

\medskip
In fact, once we have fixed a \textcolor{black}{Cartan subalgebra} in $\fg_2$ and obtained the previous decomposition, we also get a 
basis of $V_7$ consisting of weight vectors:
$$V_7=\CC e_0\oplus \CC e_{\alpha}\oplus \CC e_{-\alpha}\oplus 
\CC e_{\beta}\oplus \CC e_{-\beta}\oplus 
\CC e_{\gamma}\oplus \CC e_{-\gamma}$$
\textcolor{black}{(where $e_\alpha$ has weight $\alpha$, etc.),}
and we can normalize so that the  invariant three-form $\omega$ and the invariant 
quadratic form $Q$ have the following expressions 
in the dual basis: 
$$\omega = v_0\wedge (v_\alpha\wedge v_{-\alpha}+v_\beta\wedge v_{-\beta}+v_\gamma\wedge v_{-\gamma})
+v_\alpha\wedge v_\beta\wedge v_\gamma+v_{-\alpha}\wedge v_{-\beta}\wedge v_{-\gamma},$$
$$Q=-v_0^2+v_{\alpha}v_{-\alpha}+v_{\beta}v_{-\beta}+v_{\gamma}v_{-\gamma}.$$

\smallskip Note in particular that all our basis vectors, except $e_0$, are 
isotropic with respect to $Q$. Note also that $Q$ being non-degenerate, it allows to identify $V_7$ and $V_7^\vee$. In particular we will use the same symbol $\perp$ for the orthogonal in $V_7$ or $V_7^\vee$, with respect to $Q$ or to the canonical duality.

\setlength{\unitlength}{2mm}
\begin{picture}(20,38)(2,1)
\put(30,20){\vector(1,0){10}}
\put(30,20){\vector(1,2){4.5}}
\put(30,20){\vector(1,-2){4.5}}
\put(30,20){\vector(-1,0){10}}
\put(30,20){\vector(-1,2){4.5}}
\put(30,20){\vector(-1,-2){4.5}}

\thicklines
\put(30,20){\vector(0,1){17}}
\put(30,20){\vector(2,1){17}}
\put(30,20){\vector(-2,1){17}}
\put(30,20){\vector(0,-1){17}}
\put(30,20){\vector(2,-1){17}}
\put(30,20){\vector(-2,-1){17}}

\put(35.5,28){$\alpha$}\put(18,19.6){$\beta$}\put(35.5,11.5){$\gamma$}
\put(47.7,28){$\alpha-\beta$}\put(7.7,28){$\beta-\gamma$}\put(27.7,1){$\gamma-\alpha$}
\end{picture}

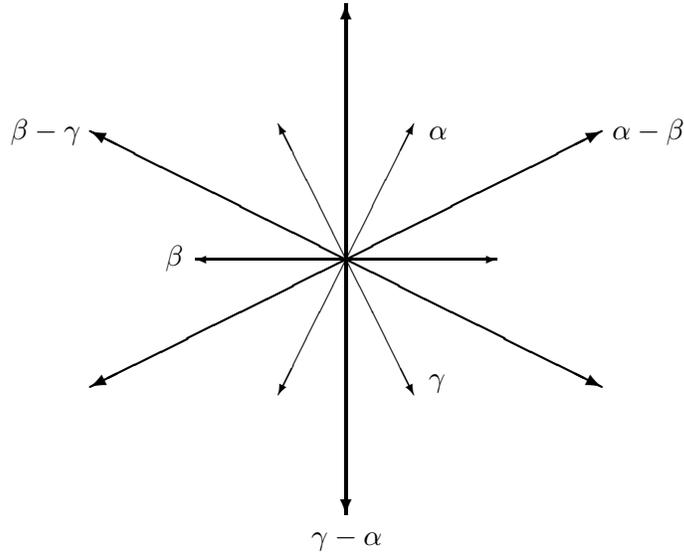
\centerline{\scriptsize{ The  root system of type $G_2$}}

\subsection{The action of $G_2$ on imaginary octonions}
The action of $G_2$ on $V_7$ provides a classical example of a prehomogeneous space. 
The following result is essentially \cite[Proposition 42]{sk}.

\begin{prop}\label{sphere} The action of $G_2$ on $\PP(V_7)$ has two orbits, 
the five dimensional quadric $\QQ^5$ defined by $Q$, 
and its complement $$\PP(V_7)-\QQ^5\simeq G_2/H,$$
where $H$ is locally isomorphic to $SL_3$.
\end{prop} 

(Here locally isomorphic means that the two groups have the same Lie algebra.) For example, the Lie algebra of the stabilizer of $[e_0]$ is the subalgebra of $\fg_2$
generated by the long root spaces: this is clearly a subalgebra,  since adding two long roots 
never gives a short root, of type $A_2$, and it certainly kills $e_0$ since no weight of $V_7$ is a long root. 

\medskip\noindent {\it Remark.}
One of the beautiful properties of the real compact form of $G_2$, which is the automorphism group of the usual octonions with real coefficients, 
is that it acts transitively on the six-dimensional sphere, seen as the set of imaginary octonions of unit norm. 
This leads to the well-known identification
of the six-dimensional sphere with the quotient space $G_2(\RR)/SU_3$, and 
to the fact that $G_2(\RR)$ can be seen
as the total space of a $SU_3$-principal bundle over $S^6$. Note that over the real numbers, the sphere is a double cover of the projective space of imaginary octonions. Over the complex
numbers, the corresponding statement is that the generic stabilizer $H$ is isomorphic to $Aut(SL_3)$, which is  a degree two extension of $SL_3$. We will not use this fact in the sequel.
\medskip

Taking orthogonals with respect to the quadratic form $Q$, Proposition \ref{sphere}
also means that  there are only two types of hyperplanes $V_6\subset V_7$ up to the action of $G_2$: they can either be transverse or tangent to the quadric. Another convenient way to distinguish them is by the restriction 
of the invariant form $\omega$. 

If $V_6=v_0^\perp$, then
$$\omega_{|V_6}=v_\alpha\wedge v_\beta\wedge v_\gamma+v_{-\alpha}\wedge v_{-\beta}\wedge v_{-\gamma}$$
belongs to the open $GL(V_6)$-orbits in $\PP(\wedge^3V_6^\vee)$.  
The two planes $\PP\langle e_{\alpha}, e_{\beta}, e_{\gamma}\rangle$ and $\PP\langle e_{-\alpha}, e_{-\beta}, e_{-\gamma}\rangle$ are uniquely determined (see e.g. \cite[3.1]{kuznetsov-c5}): this is the  {\it one apparent double point property} of the Grassmannian $G(3,V_6)$.

If $V_6=v_{-\gamma}^\perp$, then
$$\omega_{|V_6}=v_0\wedge v_\alpha\wedge v_{-\alpha}+v_0\wedge v_\beta\wedge v_{-\beta}+v_\alpha\wedge v_\beta\wedge v_\gamma$$
belongs to the tangent space to the Grassmannian $G(3,V_6^\vee)$
at the uniquely defined plane $\PP\langle v_0, v_\alpha, v_\beta
\rangle$.
(Unicity follows from the facts that, $G(3,V_6^\vee)$ being Legendrian  \cite{buczynski}, its tangent and dual varieties are identified, and that a generic singular hyperplane section is singular at a unique point.)

\medskip Denote by $\fsl_3\subset\fg_2$ the annihilator of $e_0\in V_7$. This is also the Lie algebra of the
group $H$ from Proposition \ref{sphere}.

\begin{prop}\label{compatible}
The intersection of $X_{ad}(G_2)\subset \PP(\fg_2)$ with $\PP(\fsl_3)$ is a copy of 
$F\ell_3\subset \PP(\fsl_3)$, 
the variety of complete flags in a three-dimensional space, which is also the 
adjoint variety of $\fsl_3$.
\end{prop}

\proof 
Let $V_6:=v_0^\perp$. Recall that $X_{ad}(G_2)$ is defined in $G(2,V_7)$ as the zero-locus of the section $s_\omega$ of $\cQ_5^\vee(1)$  defined by $\omega$. Since an element of the cone over $G(2,V_7)$, considered as an element of $\fso_7\simeq \wedge^2V_7$, kills $e_0$ if and only if it is contained in $G(2,V_6)$, we have 
\begin{equation}\label{X cap sl3 = X cap G(2,6)}
X_{ad}(G_2)\cap\PP(\fsl_3)=X_{ad}(G_2)\cap G(2,V_6).    
\end{equation}
Then the conclusion follows from \cite[Lemma 4]{kr}. \qed 

\medskip\noindent {\it Remark.} Recall that $F\ell_3$, being a hyperplane section of $\PP^2\times\PP^2$, has degree six. 
Also, the intersection $X_{ad}(G_2)\cap G(2,V_6)$ for $V_6^\perp$ isotropic has been described in \cite[Lemma 4]{kr} as a specific hyperplane section of the image of $P=\PP(\cO_{\PP^2}(2)\oplus T_{\PP^2}(-1))$ in $\PP^8$ by the linear system $|\cO_P(1)|$. This $P$ has also degree $6$ and is a union of projective planes parametrized by $\PP^2$.

\subsection{Action of $G_2$ on planes in $V_7$}
The action of $GL_2\times G_2$ on $\CC^2\otimes V_7$ is prehomogeneous
\cite[Proposition 43]{sk}; this implies that the action of $G_2$ on the Grassmannian  $G(2,V_7)$
has an open orbit. In this section we show that there are in fact finitely many orbits,
that we completely classify. Of course, one obvious invariant for a line in $\PP(V_7)$ is its position with respect to the invariant quadric $\QQ^5$. So we split the classification into the following two statements.

\begin{prop}\label{bisec}[Bisecant lines] Consider two points $x\ne y$ on the quadric $\QQ^5\subset\PP(V_7)$. 
Up to the action of $G_2$, there are only three possibilities.
\begin{enumerate}
    \item The line $\overline{xy}$ is a special line contained in the quadric, which means that $x$ and $y$ span a null plane.
    \item The line $\overline{xy}$ is a non-special line contained in the quadric, which means that $x$ and $y$
    span an isotropic plane which is not a null plane. 
    \item The line $\overline{xy}$ is not contained in the quadric, which means that $x$ and $y$ span a non-degenerate plane. 
\end{enumerate}
\end{prop}

\proof  It follows from \cite[Proposition 43]{sk} that the complement of the open orbit of 
$G_2$ in $G(2,V_7)$ must be irreducible. Since it certainly contains the hypersurface
of lines that are tangent to $\QQ^5$, there must be equality, which means that $G_2$ acts 
transitively on the lines that are true bisecants to the quadric.
The remaining claims follow from the general fact that there are at most two $G$-orbits of lines in
any generalized Grassmannian $G/P$, where $P$ is a maximal parabolic subgroup of a simple complex 
Lie group $G$ \cite[Theorem 4.3]{lm}.\qed

\begin{prop}[Tangent lines]
If $x\in \QQ^5$ and $\overline{xy}$ is 
a tangent line to $\QQ^5$, not contained in the quadric, then up to the action of $G_2$ there are only two possibilities. 
\begin{enumerate}
\item The linear form $\omega(x,y,\bullet)$ is a non zero multiple of $Q(x,\bullet)$.
\item These two linear forms are not proportional. 
\end{enumerate}
\end{prop}

For simplicity, we slightly abused notations in (1) by using the same letter $x$ to denote any generator of the line (idem for $y$).
Continuing with the same abuse, we note that the
first case is easier to express in the language of octonions: it means that the plane 
$\langle x,y\rangle$ contains a unique isotropic line (generated by $x$) such that the octonionic
product $xy$ is a (non-zero) multiple of $x$. We call such a plane a {\it semi-null} plane. 
Alternatively, we may observe that $\langle x, y \rangle$
is closed under the bracket on $V_7$ 
defined by $\omega$; this is the point of view chosen in \cite[Appendix A]{guseva}, see in 
particular Lemma A.5.

In the second case $xy$ is still non-zero, but
does not belong to $\langle x,y\rangle$.

\begin{proof} We may suppose that $x$ is the line generated by 
$e_\alpha$, in which case the stabilizer $P$ of $x$ in $G_2$ 
has for Lie algebra 
$$\fp = \langle \ft,  X_{\beta-\gamma}, X_{\gamma-\beta},  X_{-\beta}, X_{-\gamma}, X_\alpha,  X_{\alpha-\beta}, X_{\alpha-\gamma} \rangle \subset\fg_2,$$
where $\ft$ is our Cartan subalgebra and, as usual, for any root $\tau$  we denote by $X_\tau$ a generator
of the root space $\fg_\tau\subset\fg_2$.
The first three terms generate a copy of $\fgl_2$ which is the Levi part of this parabolic subalgebra.
The remaining terms generate the nilpotent part. 

The Zariski  tangent space to 
the quadric at $x$ identifies with  $e_\alpha^\perp/\CC e_\alpha$, in which 
there is a maximal filtration 
preserved by $\fp$:
$$\langle e_{-\beta}, e_{-\gamma} \rangle \subset \langle e_{-\beta}, e_{-\gamma}, e_0
\rangle \subset \langle e_{-\beta}, e_{-\gamma}, e_0,  e_\beta, e_\gamma \rangle.$$
Represent an element in $\PP(e_\alpha^\perp/\CC e_\alpha)$ by a vector of the form $t=v+ze_0+u$, with $u\in \langle  e_\beta, e_\gamma \rangle$ and $v\in \langle  e_{-\beta}, e_{-\gamma} \rangle$.
We claim that up to the action of $P$ there are only three orbits defined by the conditions that $u\ne 0$, or $u=0, z\ne 0$, 
or $u=0, z=0, v\ne 0$. 

If $u\ne 0$, using the action of the Levi part of $P$ we can suppose that  $u=e_\beta$. Then it is easy to check that $[t]$ belongs to the $P$-orbit of $[e_{\beta}]$, by simply using the action of the nilpotent subgroup $N$ generated by  $X_{-\beta}, X_{-\gamma}$, and of the maximal torus. We are then in case $(2)$ of the Proposition.

If $u=0$ and $z\ne 0$, the action of $N$ is enough to check that $[t]$ belongs to the $P$-orbit of $[e_0]$.  We are then in case $(1)$ of the Proposition.

Finally, if $u=0$ and $z=0$, $t$ spans a special line contained in the quadric, a situation which is excluded by  our hypothesis. 
\end{proof}

This finally yields a complete classification of the $G_2$-orbits in $G(2,V_7)$. Recall that in terms of the equivariant map $\wedge^2V_7\ra V_7^\vee\ra V_7$ deduced from the invariant three-form and quadratic form, a {\it null-plane} is characterized by the fact that a Pl\"ucker representative maps to zero, while for a {\it semi-null plane} it maps to a nonzero vector in the plane. Also, a  plane is {\it isotropic, rank one} or {\it  non-degenerate} when  the rank of the restriction of the invariant quadratic $Q$ is $0, 1, $ or $2$, respectively. 

\begin{coro}\label{planes}
Up to the action of $G_2$, a plane $V_2\subset V_7$ can be:
\begin{enumerate} 
\item a null-plane,
\item an isotropic plane which is not a null-plane,
\item a semi-null plane,
\item a rank one plane which is not semi-null,
\item a non-degenerate plane. 
\end{enumerate}
\end{coro}

Explicit representatives are, respectively:
\begin{equation}\label{representatives}
\langle e_\alpha, e_{-\beta}\rangle, \quad \langle e_\alpha, 
e_\beta\rangle, \quad \langle e_0, e_\alpha\rangle, \quad
\langle e_0, e_\alpha+e_{\beta}\rangle,\quad \langle e_\alpha, e_{-\alpha}\rangle.
\end{equation}

Each case defines a unique orbit of the $G_2$-action on $G(2,V_7)$, and the incidence diagram is

\vspace*{-4mm}
\begin{equation}\label{diagram-orbitsofplanes}
\xymatrix{& &  \mathcal{O}^7_{semi}\ar[dr] & &\\  \mathcal{O}_{gen}^{10}\ar[r] &  \mathcal{O}^9_{rank1}\ar[ur]\ar[dr] & & 
\mathcal{O}^5_{null} &\hspace{-1cm} =X_{ad}(G_2)\\ & &  \mathcal{O}^7_{iso}\ar[ur] & &}
\end{equation}

Here the superscripts are the dimensions of the orbits and an arrow from $\cO$ to $\cO'$ means that $\cO'$ is contained in the Zariski closure of $\cO$. We will see in Corollary \ref{iota}
that
$\mathcal{O}_{semi}$ and  $\mathcal{O}_{iso}$ are projectively equivalent.
By definition, the closure of $\mathcal{O}^7_{iso}$ is $OG(2,V_7)$. Following \cite[Appendix A]{guseva}, we denote by $LieGr(2,V_7)$ the closure of $\mathcal{O}^7_{semi}$.

\subsection{Decomposing five-planes}
When we restrict the invariant three-form to a codimension two subspace $V_5\subset V_7$, since a three-form on $V_5$ is the same as a two-form on $V_5^\vee$, there are three possibilities a priori: we could get zero, or a form of rank two, or a form of 
rank four. The first case is actually impossible: we know that the most degenerate situation is when $V_5$ is orthogonal to a null-plane, say $\langle e_\alpha,e_{-\beta}\rangle$, and the restriction of $\omega$ to 
$\langle e_\alpha,e_{-\beta}\rangle^\perp=\langle e_\alpha, e_{-\beta}, e_0, e_\gamma, e_{-\gamma}\rangle$ is $v_0\wedge v_\gamma\wedge v_{-\gamma}\ne 0$.

\medskip\noindent {\it Definition}. A subspace  $V_5\subset V_7$ is {\it decomposing} if $\omega_{|V_5}$ is decomposable, meaning that
there exist linear forms $v_1, v_2, v_3$ such that  $\omega_{|V_5}=v_1\wedge v_2\wedge v_3$. 
The kernel of $\omega_{|V_5}$ is a plane $A_2\subset V_5$, that we call the 
{\it axis} of the decomposing plane. 

\begin{prop}\label{decomp=seminull} 
A five-plane $V_5$ is decomposing if and only if its $Q$-orthogonal $V_2=V_5^\perp$ is a null 
or a semi-null plane in $V_7$. Its axis $A_2$ is an isotropic plane, which coincides with $V_2$ if and only 
if it is a null-plane.
\end{prop} 

\proof 
Since a skew-symmetric form has rank at most two 
if and only if its square is zero, the locus of decomposing five-planes is the zero-locus,
in $G(5,V_7)$, of the global section $t_\omega$ of $\cU_5^\vee (1)$ defined by $\omega$ 
(here $\cU_5$ denotes the rank five tautological bundle on $G(5,V_7)$). If we identify $G(5,V_7)$
with  $G(2,V_7)$ by orthogonality with respect to the invariant quadratic form, the bundle 
$\cU_5^\vee$ is identified with the quotient bundle $\cQ_5$ on $G(2,V_7)$, and   $t_\omega$
becomes a section of $\cQ_5(1)=\wedge^2\cU^\vee_2\otimes \cQ_5$. By the Borel-Weil theorem, $$H^0(G(2,V_7),\cQ_5(1))= S_{21111}V_7=Ker(V_7\otimes \wedge^5V_7\lra\wedge^6V_7),$$ 
with the usual notation for Schur functors. 
As a $G_2$-module, $\wedge^5V_7=\wedge^2V_7=\fg_2\oplus V_7$, so $V_7\otimes \wedge^5V_7$ has a unique invariant line, while $\wedge^6V_7=V_7$ has none. So $H^0(G(2,V_7),\cQ_5(1))$ 
contains a unique invariant line; 
this implies that $t_\omega$ coincides, up to scalar,
with the invariant section $t'_\omega$ of  $\cQ_5(1)$ obtained
as the composition 
$$\wedge^2\cU_2\hookrightarrow \wedge^2V_7\otimes\cO_{G(2,V_7)}\stackrel{\omega}{\lra}
V_7\otimes\cO_{G(2,V_7)}\lra \cQ_5,$$
(where we use once again the isomorphism $V_7^\vee\simeq V_7$ defined by the invariant quadratic form). 
By \cite[Lemma A.5]{guseva}, the zero-locus of $t'_\omega$ is exactly $LieGr(2,V_7)$. 
The second part of the statement can be checked explicitly, using the explicit representatives of orbits given after Corollary \ref{planes}. \qed

\medskip
From the axis $A_2$, it is actually easy to reconstruct $V_5$. Indeed, the correspondence between $A_2$ and $V_2$
induces an isomorphism between $LieGr(2,V_7)$ and $OG(2,V_7)$, as noticed in \cite[Proposition A.7]{guseva}. 
This isomorphism is actually linear and can be described as follows. Recall the decomposition into irreducible components
\begin{equation}\label{deco wedgeV7}
\wedge^2V_7=\fg_2\oplus V_7,    
\end{equation}
where $V_7\simeq V_7^\vee$ embeds into  $\wedge^2V_7$ by contraction with the (dual) invariant three-form, 
and $\wedge^2V_7$ maps to $V_7$ by a similar contraction, whose kernel is $\fg_2$. Denote by $\iota$ the 
involution of $\wedge^2V_7$ acting by $+1$ on $ \fg_2$ and by $-1$ on $V_7$. 
It obviously fixes (the cone over) $\mathcal{O}^5_{null}=X_{ad}(G_2)$, but a priori not (the cone over) $G(2,V_7)$.

\begin{prop}\label{iota}
$\mathcal{O}^7_{iso}$ and $\mathcal{O}^7_{semi}$ are projectively equivalent via $\iota$. Taking closures, we get
\[LieGr(2,V_7)=\iota(OG(2,V_7)).\]
\end{prop}

\proof
Starting from the 
Pl\"ucker representative $e_\alpha\wedge e_\beta$ of the general isotropic plane $A_2$ as above, we first 
compute its image in $V_7$ to be $\omega(e_\alpha, e_\beta)=e_{-\gamma}$. Then we identify 
this vector with the 
two-form $e_{-\gamma}.\omega=e_0\wedge e_{-\gamma}-e_\alpha\wedge e_\beta.$ As a consequence, if we apply the 
symmetry $\iota$ of $\wedge^2V_7$ defined by its decomposition into $\fg_2$-modules, we see that 
$$\iota(e_\alpha\wedge e_\beta)=e_0\wedge e_{-\gamma}.$$
This is nothing else than the Pl\"ucker representative of $V_5^\perp$, where $V_5$ is the unique decomposing five-plane
containing $A_2$, that  we computed above.\qed 

\medskip
Consider again \eqref{diagram-orbitsofplanes}. Taking closures, we get the diagram

\begin{equation*}\label{orbitsofplanes}
\xymatrix{& & LieGr(2,V_7)\ar@{_{(}->}[ld]\ar@{.>}[dd]^\iota & \\  G(2,V_7) & \; 
\mathcal{C}_3\ar@{_{(}->}[l] & & 
X_{ad}(G_2)\ar@{_{(}->}[ld]\ar@{_{(}->}[lu] \\ & &  OG(2,V_7)\ar@{.>}[uu]\ar@{_{(}->}[lu]& }
\end{equation*}
where $\mathcal{C}_3$ is a cubic hypersurface (indeed the restriction of the invariant quadratic
form to planes yields a section of $S^2\cU_2^\vee$, and the condition that this restriction is degenerate
gives a section of $\det(S^2\cU_2^\vee)=\cO_{G(2,V_7)}(3)$).
The dashed vertical arrow indicates the symmetry induced by $\iota$.

\medskip

Now we turn to the property that makes decomposing five-planes relevant for our problems. 

\begin{prop}\label{constant rank}
Let $V_5$ be a decomposing five-plane.
\begin{enumerate}
\item The kernel of the map $\omega : \wedge^2V_5\rightarrow V_7^\vee$ is a five-plane $K_5$. 
\item The intersection $X_{ad}(G_2)\cap \PP(K_5)=G(2,V_5)\cap\PP(K_5)$ is a (possibly degenerate) cubic scroll.
\end{enumerate}
\end{prop}

\begin{proof} Consider the diagram
$$\xymatrix{ & & \wedge^2V_5 \ar[d]\ar[dr]^{\omega_{|V_5}} & & \\
0\ar[r]& V_5^\perp\ar[r] & V_7^\vee\ar[r] & V_5^\vee \ar[r]& 0}$$
The kernel of the diagonal arrow is $A_2\wedge V_5$, where $A_2\subset V_5$ denotes the axis; its dimension is seven. The kernel $K_5$  of the vertical arrow is equal to the kernel of the induced map $A_2\wedge V_5 =Ker(\omega_{|V_5})\lra V_5^\perp$, in other words, we have a commutative diagram

\vspace{0.2cm}

\begin{equation}\label{diagK5}
\begin{CD}
  0  @>>>K_5    @>>> \wedge^2 V_5  @>>> V_7^\vee    \\
    @. @|      @AAA   @AAA \\
  0  @>>> K_5  @>>> A_2\wedge V_5  @>>> V_5^\perp @>>> 0   \\
\end{CD}
\end{equation}

\vspace{0.2cm}

\noindent and $K_5$ has dimension at least five. If this dimension was bigger, the intersection $X_{ad}(G_2)\cap G(2,V_5)$ would contain a hyperplane section of $\PP(A_2\wedge V_5)\cap G(2,V_5)$. 
The latter intersection is the set of planes in $V_5$ that meet $A_2$ non trivially; geometrically, this is a cone over a cubic scroll isomorphic to $\PP^1\times\PP^2$. 
But a hyperplane section of such a cone always contains a plane, and we would conclude that  $X_{ad}(G_2)$ contains a plane, which is not the case \cite[Lemma 3]{kr}. This proves (1), and then 
$X_{ad}(G_2)\cap \PP(K_5)=G(2,V_5)\cap\PP(K_5)$ is a codimension 
two linear section of a cone over $\PP^1\times\PP^2$, hence 
a cubic scroll by \cite[Lemma 5]{kr}. More precisely, one easily checks that the scroll is smooth when the section does not contain the vertex, which means that $A_2$ is not a null-plane;  otherwise it is a cone over a twisted cubic. This proves (2). 
\end{proof}

\textit{Remark.} 
Consider a decomposing five-plane $V_5$, so that $\omega_{|V_5}=e_1\wedge e_2\wedge e_3$. 
If $V_5=\langle e_6, e_7\rangle^\perp, $ there exists two-forms $\theta_6$ and $\theta_7$ such that 
$$\omega =e_1\wedge e_2\wedge e_3+e_6\wedge \theta_6+e_7\wedge \theta_7,$$
and the axis is $A_2=\langle e_1, e_2, e_3, e_6, e_7\rangle^\perp$.
For $\psi\in \wedge^2V_5$, the contraction with $\omega$ yields 
$$\iota_\psi(\omega)= \iota_\psi(e_1\wedge e_2\wedge e_3)+\iota_\psi(\theta_6) e_6+\iota_\psi(\theta_7) e_7,$$
showing that the image of $\omega : \wedge^2V_5\rightarrow V_7^\vee$ is contained in $A_2^\perp$, hence it is equal to it.

\smallskip The fact that an intersection $X_{ad}(G_2)\cap G(2,V_5)$ is always a conic or a cubic scroll, possibly singular, was already observed in \cite[Lemma 5]{kr}. We can be more precise and compute 
this intersection for the five types of five-planes orthogonal to the
five types of planes 
given in Corollary \ref{planes}; using the explicit representatives of (\ref{representatives}), this is a straightforward computation whose result is presented in the table below.
\smallskip

\begin{table}[ht]\label{table} 
\caption{}
\begin{tabular}{|l|l|l|}
\hline
$V_2=V_5^\perp$ & Representative & $X_{ad}(G_2)\cap G(2,V_5)$\\
\hline
\textit{non-degenerate} & $\langle e_{\alpha}, e_{-\alpha}\rangle$ & smooth conic\\
\textit{rank one} & $\langle e_0, e_{\alpha}+e_\beta\rangle$  & reduced singular conic\\
\textit{isotropic} & $\langle e_{\alpha}, e_{\beta}\rangle$ & double line \\
\textit{semi-null} & $\langle e_{0}, e_{\alpha}\rangle$ & smooth cubic scroll \\
\textit{null} & $\langle e_{\alpha}, e_{-\beta}\rangle$ & cone over a rational cubic \\
\hline
\end{tabular}
\end{table}

\medskip
Let us briefly treat the isotropic and semi-null cases; the other ones are obtained with similar explicit computations and will be left to the reader. \smallskip

If $V_2=\langle e_\alpha, e_\beta\rangle$, then $V_5=V_2^\perp=
\langle e_\alpha, e_\beta, e_0, e_\gamma, e_{-\gamma}\rangle$. We deduce that 
$$\wedge^2V_5\cap \fg_2 =\langle e_\alpha\wedge e_{-\gamma}, \;
e_\beta\wedge e_{-\gamma}, \; e_\alpha\wedge e_\beta-e_0\wedge e_{-\gamma}
\rangle .$$ 
For a form $\theta =x e_\alpha\wedge e_{-\gamma} + ye_\beta\wedge e_{-\gamma}
+z (e_\alpha\wedge e_\beta-e_0\wedge e_{-\gamma})$ in this three-plane
we have $\theta\wedge\theta =-2z^2e_\alpha\wedge e_\beta\wedge e_0\wedge 
e_{-\gamma}$, hence the double line. \smallskip

If $V_2=\langle e_{0}, e_{\alpha}\rangle$ then  $V_5=V_2^\perp=
\langle e_{\alpha}, e_\beta, e_{-\beta}, e_\gamma, e_{-\gamma}\rangle$ and 
$$\wedge^2V_5\cap \fg_2 =\langle e_\alpha\wedge e_{-\beta}, \;
e_\alpha\wedge e_{-\gamma}, \; e_\beta\wedge e_{-\gamma}, \;
e_\gamma\wedge e_{-\beta}, \; e_\beta\wedge e_ {-\beta}-e_\gamma\wedge e_{-\gamma}
\rangle .$$
For $\theta =pe_\alpha\wedge e_{-\beta}+q
e_\alpha\wedge e_{-\gamma}+re_\beta\wedge e_{-\gamma}+s 
e_{-\beta}\wedge e_\gamma+t(e_\beta\wedge e_ {-\beta}-e_\gamma\wedge e_{-\gamma})$,
we compute that $\frac{1}{2}\theta\wedge\theta$ is equal to
$$(pr-qt)e_\alpha\wedge e_{-\beta}\wedge 
e_\beta\wedge e_{-\gamma}+(qs-pt)e_\alpha\wedge e_{-\gamma}\wedge e_{-\beta}\wedge e_\gamma +(rs-t^2)e_\beta\wedge e_{-\gamma}\wedge e_{-\beta}\wedge e_\gamma.$$
So $\theta\wedge\theta=0$ if and only if the rank of the matrix  $\begin{pmatrix} p&s&t\\q&t&r
\end{pmatrix}$ is at most one, hence the smooth cubic scroll.

\subsection{Some vector bundles}\label{svb}
By Proposition \ref{iota}, $LieGr(2,V_7)$ has two  $G_2$-equivariant embeddings
in $G(2,V_7)$, the natural one and its twist by $\iota$. 

\smallskip\noindent {\it Definition}. To keep the exposition readable, we use the same letter for the tautological bundle $\cU_2$ of $G(2,V_7)$ and its restriction to $LieGr(2,V_7)$.
We denote by $\cV_5=\cU_2^\perp$ 
the bundle of decomposing five-planes (orthogonality being taken with respect to the invariant quadratic form $Q$). The rank two subbundle $\cA_2\subset \cV_5$ defined by their axis (recall Proposition \ref{decomp=seminull}) is the tautological rank two bundle on $LieG(2,V_7)$ corresponding to the twisted embedding in $G(2,V_7)$; in the terminology of \cite[Remark A.8]{guseva}, $\cA_2$ is the dual spinor bundle.

\smallskip
By definition of decomposing five-planes and their axis,  
$\omega$ descends to a global section of $\wedge^3(\cV_5/\cA_2)^\vee$, everywhere nonzero. 
Also recall that  the morphism $\wedge^2\cV_5\ra V_7^\vee\otimes\cO_{LieGr(2,V_7)}$
has constant rank by Proposition \ref{constant rank}.

\medskip\noindent {\it Definition}. We denote by  $\mathcal{K}_5$
the kernel of this morphism, a rank five vector  bundle on 
$LieGr(2,V_7)$. 

\smallskip 
This bundle  admits a natural inclusion
\begin{equation}\label{K5}
\cK_5\hookrightarrow \fg_2\otimes\cO_{LieGr(2,V_7)}.
\end{equation}
All these bundles are  $G_2$-equivariant and, taking into account the remark after Proposition \ref{constant rank}, the relative version of diagram (\ref{diagK5}) can be completed into the following 
commutative diagram of vector bundles:

\small
\begin{equation}\label{vb-diagram}
\begin{CD}
  @.    @. 0   @. 0 @. \\
  @.   @. @AAA        @AAA   \\
  @.  @.\wedge^2(\mathcal{V}_5/\cA_2)   @=  (\mathcal{V}_5/\cA_2)^\vee   @. \\
   @.  @. @AAA        @AAA   \\
  0  @>>>\mathcal{K}_5    @>>> \wedge^2\mathcal{V}_5  @>>> \mathcal{A}_2^\perp   @>>> 0 \\
    @. @|      @AAA   @AAA \\
  0  @>>> \mathcal{K}_5  @>>> \cA_2\wedge\mathcal{V}_5  @>>> \cU_2   @>>> 0 \\
   @.    @. @AAA        @AAA  \\
    @.    @. 0   @. 0   @. 
\end{CD}
\end{equation}
\normalsize
\medskip 

Since $\wedge^3(\cV_5/\cA_2)$ is trivial, $\cV_5$ and 
$\cA_2$ have the same determinant, and their orthogonals as well. We deduce that
\begin{equation}\label{det}
\begin{array}{rcl}
\det(\cU_2)= \det(\mathcal{A}_2)= \det(\mathcal{V}_5) &= & \cO_{LieGr(2,V_7)}(-1), \\ 
\det(\mathcal{K}_5) &=& \cO_{LieGr(2,V_7)}(-3).
\end{array}
\end{equation}

\medskip Recall from the proof of Proposition \ref{constant rank}
that the cubic scroll associated to a decomposing five-plane $V_5$
is a codimension two linear section, by $\PP(K_5)$, of the Schubert cycle of planes in $V_5$ that meet its axis $A_2$. This Schubert 
cycle is contained in $\PP(A_2\wedge V_5)$, and can be identified with a cone over $\PP(A_2)\times\PP(V_5/A_2)$ with vertex 
$\PP(\wedge^2A_2)$. The linear section is a cone over a rational cubic when it contains the vertex (meaning that $A_2$ is a null plane), otherwise it is a smooth cubic scroll. 

For future use (more precisely, for section \ref{section bpf}), let us denote by $\cR$ and $\cC$ the corresponding families of cubic scrolls  and Schubert cycles. They fit in the following  diagram:

\begin{equation}\label{RC}
\xymatrix{& \PP_{LieGr(2,V_7)}(\cA_2\wedge\cV_5)\ar[rr] && \PP(\wedge^2V_7) \\
\cC\;\ar@{^{(}->}[ru]\ar[rr]  && G(2,V_7)\ar@{^{(}->}[ru]  & \\ 
&  \PP_{LieGr(2,V_7)}(\cK_5)\ar@{^{(}->}[uu]\ar[rr] && \PP(\fg_2)\ar@{^{(}->}[uu] \\
\cR\;\ar@{^{(}->}[uu]\ar@{^{(}->}[ru]
\ar[rr]
&& X_{ad}(G_2)\ar@{^{(}->}[uu]\ar@{^{(}->}[ru] & 
}
\end{equation}

\smallskip
The whole lefthand part of this diagram is fibered over $LieGr(2,V_7)$. 

Moreover, $\cR$ is defined in $\cC$ as the intersection with $\PP_{LieGr(2,V_7)}(\cK_5)$. This will be used in the proof of  Theorem \ref{basepointfree}.

\subsection{Cubic scrolls in hyperplane sections of $X_{ad}(G_2)$}\label{sect cubicscrolls}

Let us come back to diagram (\ref{fundamental}). 
We denote by $\Upsilon(X_{ad}(G_2))$ the Hilbert scheme of cubic scrolls and cones in $X_{ad}(G_2)$.

\begin{lemma}\label{scrolls}
For any line $L\subset\QQ^5$, the scheme $R_L:=p_X(p_\QQ^{-1}L)$ is a cubic scroll or cone in $X_{ad}(G_2)$. Moreover the map 
$$OG(2,V_7)\ra \mathrm{Hilb}(X_{ad}(G_2)), \qquad L\mapsto R_L,$$
is an isomorphism onto $\Upsilon(X_{ad}(G_2))$. 
\end{lemma}

\proof
The correspondence  (\ref{fundamental}) allows to identify $\QQ^5$ with the Hilbert scheme 
$F_1(X_{ad}(G_2))$ of lines in $X_{ad}(G_2)$.
Recall from Proposition \ref{bisec} that there are two types of lines in $\QQ^5$. If $L$ is special, it defines a point $x$ in $X_{ad}(G_2)$, and $R_L$ is the union of lines in $X_{ad}(G_2)$ passing through that point - hence a cubic cone. If $L$ is not special, then $R_L$ is still a surface of degree three.
if $R$ is a cubic scroll or cone, the pencil of lines in the ruling of $R$ defines a curve $C\in \QQ^5$, which for degree reasons must be a line. \qed

\medskip 

Let $\mathcal{X}\subset X_{ad}(G_2)\times \PP(\fg_2^\vee)$ be the universal hyperplane section of $X_{ad}(G_2))$, and  $\Upsilon(\cX/ \PP(\fg_2^\vee))$ 
be the relative Hilbert scheme of cubic scrolls and cones. There is a diagram
\begin{equation}\label{relative scrolls}
\xymatrix{
 & \Upsilon(\cX/ \PP(\fg_2^\vee))\ar@{->}[ld]_{p_\Upsilon}\ar@{->}[rd]^{p_\fg} & \\
 \Upsilon(X_{ad}(G_2)) & & \PP(\fg_2^\vee),
}
\end{equation}
and  $\Upsilon(X_{ad}(G_2)\cap H)=p_\fg^{-1}([H])$ for any hyperplane $H\subset\fg_2$. 
The following information can be extracted from \cite{kuznetsov-hpd}. 
\begin{enumerate}
\item  $p_\Upsilon$ is a $\PP^8$-bundle by \cite[Lemma 8.5]{kuznetsov-hpd} (the scheme $\Sigma'$ defined there is the universal cubic scroll or cone by Lemma \ref{scrolls}).
\item  $\Upsilon(\cX/ \PP(\fg_2^\vee))\simeq \tilde{Y}$  in the notation of \cite{kuznetsov-hpd}, and the map $p_\fg$ coincides with $\tilde{g}$ therein.
\item  The Stein factorization of this map takes the form $p_\fg=\tilde{g}= \bar{g}_3\circ \phi$ by \cite[Lemma 8.8]{kuznetsov-hpd}.
\item
The dual of the decomposition
\eqref{deco wedgeV7} yields a projection $\PP(\wedge^2 V_7^\vee)\dashrightarrow \PP(\fg_2^\vee)$. The image of $G(2,V_7^\vee)$ via the projection 
is $Z_1$ in Kuznetsov's notation, see \cite[Lemma 8.2 and Corollary 8.11]{kuznetsov-hpd}.
Over the complement of $Z_1$, the map $\phi$ 
is an étale-locally trivial $\PP^2$-bundle by \cite[Proposition 8.9]{kuznetsov-hpd}, and the map $\bar{g}_3: \tilde{Y}_3\lra \PP(\fg_2^\vee)$ is a
double cover branched along the dual variety of $X_{ad}(G_2)$, i.e. the discriminant sextic hypersurface $\Delta\subset \PP(\fg_2^\vee)$.
\end{enumerate}

Since $Z_1$ is irreducible, a union of $G_2$-orbits and 10-dimensional, it has to be the closure of $O_{10}$ in \cite[Lemma 3.1]{kapustka-mukai}.
The intersection $X_{ad}(G_2)\cap H$ is smooth for $[H]\notin \Delta$,  and nodal with a single node for $[H]\in \Delta\smallsetminus Z_1$ (\cite[Proof of Theorem 2.3]{kapustka-mukai}), while $Z_1$ parametrizes highly singular hyperplane sections.

Putting everything together, we deduce the following statement:

\begin{prop}\label{two veronese}
The Hilbert scheme $\Upsilon(X_{ad}(G_2)\cap H)$  is isomorphic to
\begin{enumerate}
    \item $\PP^2\sqcup\PP^2$ if $[H]\notin\Delta$, i.e. $X_{ad}(G_2)\cap H$ is smooth;
    \item $\PP^2$ if $[H]\in\Delta\smallsetminus Z_1$, i.e. $X_{ad}(G_2)\cap H$ is nodal.
    \end{enumerate}
    \end{prop}

The first part of the statement can also almost be deduced directly from Lemma \ref{scrolls}, as follows. An equation of $H$ defines a section of the dual Cayley bundle $\cC^\vee$, and the zero locus $Z$ of this section is the locus of points in $\QQ^5$ defining lines in $X_{ad}(G_2)$ that are contained in $H$. So cubic scrolls and cones contained in $X_{ad}(G_2)\cap H$ are in correspondence with lines in $Z$. But by \cite[Theorem 3.7]{ottaviani-cayley}, when $Z$ is smooth it is isomorphic to the flag manifold $F\ell_3$, hence lines in $Z$ are parametrized by 
two copies of $\PP^2$. \smallskip

Cubic surface scrolls in smooth hyperplane sections $F=X_{ad}(G_2)\cap H$ were also studied in \cite{kr} by reducing to surface scrolls in $G(3,U_6)$. Indeed, fix a vector space $U_6$ of dimension six. By \cite[Theorem 1]{kr}, there exists a unique linear embedding $\psi:H\hookrightarrow \PP(\bigwedge^3 U_6)$ sending $F$ isomorphically to a four dimensional linear section of $G(3,U_6)$.
Now, inside $G(3,U_6)$ there are two different types of cubic scrolls, yielding two families in $F$ both parametrized by $\PP^2$. 

When $F$ is nodal with a single node $f$, the projection $\psi$ from the node sends $F$ birationally to 
a codimension two linear section $\tilde{F}=LG(3,U_6)\cap P$ of 
the Lagrangian Grassmannian $LG(3,U_6)$, see \cite{kapustka-mukai}. The cubic scrolls and cones in $F$ are sent to maximal quadrics in $\tilde{F}$, and one can check that such quadrics are parametrized by a projective plane.

\medskip \noindent {\it Remark.}
Not surprisingly, the cubic scrolls in $X_{ad}(G_2)$ defined by decomposing five-planes in $V_7$ also yield twisted cubic curves in Fano fourfolds and threefolds of genus ten. As observed by an anonymous referee,  Proposition \ref{two veronese} is closely related to Theorem 9.11 of \cite{kuzprokh}, which asserts that the Hilbert scheme of rational cubic curves inside a smooth prime Fano threefold of genus ten is a $\PP^2$-bundle over a genus two curve.

 \section{The involution}
In all the sequel we consider a general $K3$ surface $S$ of genus $10$. (Beware that $S$ will be very general starting from section \ref{sec J}.) By the work of Mukai \cite{mukai}, $S$ can be  described as a generic codimension three linear section of the adjoint variety of $G_2$, 
\begin{equation}\label{Mukai model}
S=X_{ad}(G_2)\cap L \subset \PP(\fg_2),
\end{equation}
where $L=\PP(V_{11})$ for $V_{11}\subset\fg_2$ a codimension three subspace.
\smallskip

We will first present our geometric construction of the birational involution $\varphi$ of $S^{[3]}$ whose existence is known abstractly from \cite{beri}. This will immediately stress the rôle of the variety $J$ parametrizing schemes in $S^{[3]}$ that do not generate a hyperplane. These subschemes of $S$ are closely related to the cubic scrolls we studied in the previous section. In order to get a better understanding of the indeterminacy locus $I$ of $\varphi$, we will then need to dig a little deeper in Bayer-Macri's description of the Mori cone and of wall-crossing for $S^{[3]}$. We will deduce a description of $I$ in terms of vector bundles, implying that $I=J$ and that it admits a structure of $\PP^2$-fibration. This is exactly what allows to define a Mukai flop, and we will finally prove that this Mukai flop is nothing else that $\varphi$. 

\medskip\noindent {\it Hypothesis.} In the next section \ref{section invol}, $S$ will be any smooth polarized K3 surface of genus ten admitting a Mukai model (\ref{Mukai model}). According to Mukai, this is equivalent to asking that the polarized K3 surface is Brill-Noether general 
\cite[Theorem 3.10 and Theorem 4.7]{mukai-bn} (but see also \cite[Section 1.2-1.3]{bkm2}). 
Polarized K3 surfaces which are Brill-Noether general fill a Zariski open subset of  the 19-dimensional  moduli space of K3 surfaces of genus ten, whose complement has been explicitly computed \cite[Lemma 2.8]{greer}.

Starting from section \ref{sec J} we will 
always suppose that $S$ is very general, in the sense that $Pic(S)$ is generated by the hyperplane divisor (the restriction of the hyperplane divisor from $\PP(\fg_2)$).  We will  use more than once that this implies that $S$ contains no line and no conic (more generally, no curve of degree smaller than $18$). 
For statements like Lemma \ref{cork psi}, we will also need $L$ to be generic enough for certain degeneracy schemes to be well-behaved. Some specific genericity assumptions are also required in order to be able to apply the results of \cite{kr}. In the end we will not attempt to make our genericity conditions precise.

\subsection{Construction of the involution}\label{section invol}
 Let $S$ be a smooth K3 surface obtained as a linear section (\ref{Mukai model}) of $X_{ad}(G_2)$.

\begin{lemma}\label{general}
Let $p_1,p_2,p_3$ be generic points on $S$. They correspond to three null planes $P_1,P_2,P_3$ spanning a hyperplane of $V_7$.
\end{lemma}

\proof By contradiction, suppose that  $P_1,P_2,P_3$ are contained in  a five-plane $V_5$. Since $S$ is a linear section of  $X_{ad}(G_2)$, Table 1 shows that $V_5$ must be decomposing, since otherwise we would conclude that there is a conic or a line in $S$  passing through the generic point, which is not the case. 

Now, fix $P_1$ and $P_2$, spanning a four-plane $V_4$. Recall that by the generalized Bruhat decomposition (see e.g. \cite[Proposition 3.16]{borel-tits}), the relative positions of two points in a rational homogeneous space are in bijection with double cosets of the Weyl group. Applying this to $X_{ad}(G_2)$, we conclude that, up to the action of $G_2$, there are only four possible relative positions for two points $p_1$ and $p_2$ of the adjoint variety. These relative positions  are easy to identify: either $p_1$ and $p_2$ coincide, or they are joined by a line in $X_{ad}(G_2)$, or they are orthogonal with respect to the Killing form, or they are in general position.

A dimension count shows that in the latter case, there is only a one-dimensional family of decomposing five-planes $V_5$ containing $V_4$.
If $V_5$ is constant when we move the third point $P_3$, the surface $S$ must be contained in $G(2,V_5)$, which is clearly absurd by Table 1 again. 
If $V_5$ moves in a one-dimensional  family, then $S$ is covered by its intersections with the corresponding $G(2,V_5)$'s. In particular these intersections must be curves, and linear sections of the cubic scrolls or cones $X_{ad}(G_2)\cap G(2,V_5)$; so they are contained in hyperplane sections of these scrolls, hence they must be rational curves or union of rational curves. We would thus conclude again that $S$ is covered by rational curves,  which is not the case. 

We conclude that $p_1$ and $p_2$ cannot be in general position in $X_{ad}(G_2)$, which means that they are orthogonal with respect to the Killing form. Since  this must be true for any two general points in $S$, this implies that the span $L$ of $S$ is isotropic; a contradiction since the Killing form is non-degenerate, and $L$ has dimension bigger than $6$. 
\qed

\medskip 
For $p_1,p_2,p_3$ generic points on $S$, let $V_6$ denote the span of the  corresponding  null planes in $V_7$.
By Proposition \ref{compatible} and \eqref{X cap sl3 = X cap G(2,6)},
$$\begin{array}{rcccl}
S\cap G(2,V_6) & = & (X_{ad}(G_2)\cap L)\cap G(2,V_6) & = & \hspace*{3cm} \\
\hspace{3cm} &=& (X_{ad}(G_2)\cap G(2,V_6))\cap L &= & F\ell(V_6)\cap L
\end{array}$$
is a codimension three section of the flag variety, hence in general consists in six reduced points: $p_1,p_2,p_3$ 
plus three other points $q_1,q_2,q_3$. Let
$$\varphi(p_1+p_2+p_3)=q_1+q_2+q_3. $$
This gives a rational map from $S^{[3]}$ to itself, more 
formally described in the proof of the next statement.

\begin{theorem}\label{involution}
Let $S=X_{ad}(G_2)\cap L$ be a K3 surface of genus 10. Then there is a unique
birational involution 
$\varphi:S^{[3]}\dashrightarrow S^{[3]}$
such that for a subscheme $Z\subset S$ defined by a general point of $S^{[3]}$, there is a unique hyperplane $V_6\subset V_7$ such that 
\[
S\cap G(2,V_6)=Z\cup \varphi(Z). 
\]
\end{theorem}

\medskip

In particular, for $S$ very general $\varphi$ has to coincide with the non symplectic birational involution whose 
existence was predicted in \cite{beri}.

\proof
Let  
$\cQ_5$ denote the quotient bundle on the Grassmannian $G(2,V_7)$, and consider the natural morphism $\PP_S(\cQ_5^\vee)\lra \PP(V_7^\vee)$. Since the pre-image of a point defined by a hyperplane $V_6$ is the set of pairs $(P\subset V_6)$ where the plane $P$ defines a point of $S$, this morphism is generically $6$ to $1$. Therefore, the relative Hilbert scheme $Hilb^3(\PP_S(\cQ_5^\vee)/\PP(V_7^\vee))$ admits an obvious non-trivial birational involution. Finally, the projection of $\PP_S(\cQ_5^\vee)$ to $S$ induces a birational morphism from $Hilb^3(\PP_S(\cQ_5^\vee)/\PP(V_7^\vee))$ to $S^{[3]}$, whose inverse associates to a triple of general points of $S$ the hyperplane generated by the three
corresponding planes. It is clear that the induced birational involution of $S^{[3]}$ coincides with $\varphi$.
\qed

\subsection{Schemes generating a non-maximal linear space.}\label{sec J}
From now on, $S$ is a very general linear section of $X_{ad}(G_2)$.
The definition of $\varphi$ stresses the condition, for three 
points in $S$, that the corresponding planes in $V_7$ generate a hyperplane. 
In order to understand better this condition, 
consider  the universal 
scheme $\cZ$ of length three over $S$, with the two projections 
\begin{equation}\label{universal-scheme}
\xymatrix{ & \cZ\ar[ld]_p\ar[rd]^q & \\ S^{[3]} & & S.}
\end{equation}
Denote again by $\cU_2$ the restriction to $S$ of the rank two tautological bundle on $G(2,V_7)$. 
The tautological epimorphism $V_7^\vee\otimes \cO_S\lra \cU_2^\vee$ 
induces a morphism
\begin{equation}\label{morphism-phi}
\phi: V_7^\vee\otimes \cO_{S^{[3]}}\lra p_*(q^*\cU_2^\vee).
\end{equation}
Since $p$ is flat, $p_*(q^*\cU_2^\vee)$ is a vector bundle of rank six, sometimes denoted  $(\cU_2^\vee)^{[3]}$.
We say that a scheme $Z\subset G(2,V_7)$ generates the linear space $W\subset V_7$ if
$$W^\perp= Ker(V_7^\vee \lra H^0(Z,{\cU_2^\vee}_{|Z})).$$

Then of course $Z\subset G(2,W)$, since this sub-Grassmannian is cut out in $G(2,V_7)$ by the linear forms from $W^\perp\wedge V_7^\vee\subset \wedge^2V_7^\vee$. 

\medskip \noindent {\it Definition.} 
Let $J\subset S^{[3]}$ 
denote the degeneracy locus of $\phi$, with the scheme
structure defined by the corresponding Fitting ideal. This is the locus of length three subschemes of $S$ that fail to generate a hyperplane of $V_7$.

\medskip
When $Z$ is a reduced length three subscheme of $S$, outside $J$, 
it consists in three points corresponding to three planes in general position in $V_7$, whose linear span is a hyperplane. By Lemma \ref{general}, $J$ is a proper subset of $S^{[3]}$.

The locus $J\subset S^{[3]}$ is closely related to decomposing five-planes, as our next result shows.

\begin{prop}\label{Zdec}
Let $Z\subset S$ be a length three subscheme that does not generate a hyperplane. Then it generates a decomposing 
five-plane in $V_7$.  
\end{prop}

\proof First observe that $Z$ cannot generate a subspace of $V_7$ of dimension smaller than five. Indeed, for any $V_4\subset V_7$ the intersection $X_{ad}(G_2)\cap G(2,V_4)=\PP(\fg_2)\cap G(2,V_4)$ is a conic, or a line, or a finite scheme of length at most two. Since $S$ contains no conic and no line,
a fortiori $S\cap G(2,V_4)$  must be a finite scheme of length at most two, and cannot contain $Z$.

If $Z$ generates a $V_5$, it must be contained inside $G(2,V_5)\cap X_{ad}(G_2)$. But recall from Table \ref{table}
that this is a conic when $V_5$ is not decomposing,
and by the same argument as before  $S\cap G(2,V_5)$ cannot contain 
$Z$.\qed

\medskip
The previous result and Proposition \ref{decomp=seminull} imply that there is a map \[\alpha:J\lra LieGr(2,V_7)\]
associating to a subscheme $Z$ the decomposing five-plane $V_5$ it generates.
The map is injective, since we recover $Z$ from $V_5$ by intersecting with 
$L=\PP(V_{11})$ the cubic surface scroll $X_{ad}(G_2)\cap G(2,V_5)$.
Since $L\subset \PP(\fg_2)$ has codimension three, the intersection cannot be transverse: we will detect this defect 
of transversality through the morphism  
\begin{equation}\label{psi}
    \psi: \cK_5\lra (\fg_2/V_{11})\otimes \cO_{LieGr(2,V_7)}
\end{equation}
deduced from (\ref{K5}). We denote by $\cD(\psi)$ its first degeneration scheme.

\begin{lemma}\label{JisDpsi}
$\cD(\psi)=\alpha(J)$.
\end{lemma}

\proof 
Let $V_5$ be a decomposing five-plane, and suppose that the corresponding morphism 
$K_5\lra \fg_2/V_{11}$ is surjective, so that its kernel is two-dimensional. Then by 
Proposition \ref{constant rank} and Table \ref{table} the subscheme
$$S\cap G(2,V_5)=\PP(V_{11})\cap X_{ad}(G_2)\cap G(2,V_5)$$
is the intersection of a cubic scroll (or cone) with a line,  hence it has length at 
most two since $S$ contains no line. 
So $\alpha(J)$ is contained in $\cD(\psi)$. 

Conversely, if $V_5$ belongs to $D(\psi)$, $L=\PP(V_{11})$ meets $\PP(K_5)$ along a codimension two subspace, that cuts the (possibly
degenerate) cubic scroll $X_{ad}(G_2)\cap G(2,V_5)$ along a length three subscheme of $S$ (once again, this intersection cannot be 
positive dimensional, since it is defined by quadrics and $S$ contains no line or conic.)
\qed 

\begin{lemma}\label{cork psi}
$\psi$ has constant corank one over $\alpha(J)$ and $\alpha$ is an isomorphism onto its image. In particular, $J$ is smooth.
\end{lemma}

\begin{proof}
If $\PP(K_5)\cap L$ has dimension three, $L$ cuts the corresponding cubic scroll along a degree $3$ curve which is contained inside $S$. But there is no such curve in $S$.
Since $\alpha(J)$ is a degeneracy locus by Lemma \ref{JisDpsi}, for $L$ generic it is singular where the rank of $\psi$ drops by more than one, which never happens.
So $\alpha(J)$ is smooth, and we conclude by injectivity of $\alpha$.
\end{proof}

In words, we conclude that $J$ parametrizes pairs 
$(V_5,H)$, where $V_5\subset V_7$ is a decomposing five-plane and 
$H\subset \PP(\fg_2)$ is a hyperplane through $S$ that also 
contains the cubic scroll $X_{ad}(G_2)\cap G(2,V_5)$; moreover this hyperplane is unique, hence there is a morphism
\begin{equation}\label{J}
    \pi : J \lra L^\perp\simeq\PP^2.
\end{equation}

We are now ready to describe the structure of $J$, using results from Section \ref{sect cubicscrolls}.
Because of \cite[Lemma 7]{kr}, the fiber of $\pi$ over the hyperplane $H$ can be identified 
with the 
Hilbert scheme of cubic surface scrolls in $X_{ad}(G_2)\cap H$, so
\begin{equation}\label{union}
    J=\bigcup_{[H]\in L^\perp} \Upsilon(X_{ad}(G_2)\cap H).
\end{equation}

Let $C$ be the intersection of $L^\perp \subset\PP(\fg^\vee_2)$ with the dual variety $\Delta$ of 
$X_{ad}(G_2)$; since $L$ is generic, this is a smooth plane sextic curve.
Denote by $g:\Sigma\lra L^\perp\simeq\PP^2$  the double cover branched over the sextic $C$, so that $\Sigma$ is a smooth K3 surface.

\begin{coro}\label{J irred}
$J$ is irreducible. The Stein factorization of $\pi: J\lra L^\perp$ is
\begin{equation}\label{stein for J}
    J  \xrightarrow{f} \Sigma \xrightarrow{g} L^\perp,
\end{equation}
where $f$ is an étale locally trivial $\PP^2$-fibration.
\end{coro}

\begin{proof} 
Via \eqref{union}, the map \eqref{J} is nothing but the restriction of $p_g$ in \eqref{relative scrolls} over $L^\perp\subset \PP(\fg^\vee_2)$. Generically $L^\perp\cap Z_1=\emptyset$, then the Stein factorization follows from Proposition \ref{two veronese}. 
\end{proof}

The  surface $\Sigma$ is the K3 surface of degree two which is associated to $S$ (together with a Brauer class), 
in the context of Homological Projective Duality for the
adjoint variety of $G_2$ \cite[section 6.4]{kuznetsov-hpd}.

Moreover, in \cite{kr} the authors give an explicit description of $\Sigma$ in terms of stable vector bundles on $S$. Indeed, denote by $\mathcal{E}^\sigma$ and $\mathcal{E}^\tau$, respectively, the dual of the tautological bundle and the tautological quotient bundle over $G(3,U_6)$, both of rank three; over the Lagrangian Grassmannian $LG(3,U_6)$ the two bundles are isomorphic, we denote them by $\widetilde{\mathcal{E}}$.

For $[H]\in L^\perp$, so that  $F=X_{ad}(G_2)\cap H$ contains $S$, recall from  Section \ref{sect cubicscrolls} that there is a regular map $\psi_H:  F\lra G(3,U_6)$ when $F$ is smooth, and a rational map $\psi_H:  F\dashrightarrow LG(3,U_6)$ when $F$ is nodal. Their restrictions $\psi_S$ to $S$ 
are always regular. \cite[Theorem 2]{kr} and its proof yield the following statement.

\begin{theorem}\label{desc Sigma kr}
$\Sigma$ is isomorphic with the moduli space $M_\mathcal{H}(3,\mathcal{H},3)$ of $\mathcal{H}$-stable vector bundles with Mukai vector $(3,\mathcal{H},3)$.
The fiber over $[H]$ of the double cover $g: \Sigma\lra L^\perp$ is
\begin{itemize}
    \item 
    the classes of the two bundles $\mathcal{E}_H^\sigma=\psi_S^*\mathcal{E}^\sigma$ and $\mathcal{E}_H^\tau=\psi_S^*\mathcal{E}^\tau$ when $[H]\notin C$,
    \item the class of the bundle $\psi^*_S \widetilde{\mathcal{E}}$ when $[H]\in C$.
\end{itemize}
\end{theorem}

\smallskip \noindent {\it Definition.} 
We say that a cubic surface scroll in $G(3,U_6)$ is \emph{of $\sigma$-type} if it is contained in the Schubert variety $Fl(U_1,3,U_6)$ of spaces containing some line $U_1\subset U_6$.
It is \emph{of $\tau$-type} if it is contained in the Grassmannian $G(3,U_5)$ for some $U_5\subset U_6$. (Note that 
$Fl(U_1,3,U_6)$ and $G(3,U_5)$ are both isomorphic with $G(2,5)$.)

\medskip

When $[H]\notin C$, both $\mathcal{E}_H^\sigma$ and $\mathcal{E}_H^\tau$ have a six-dimensional space of global sections: in the projectivizations of those spaces lie the two copies of $\PP^2$ (two Veronese surfaces, see \cite[Proposition 4]{kr}) parametrizing the fiber of $f:J\lra \Sigma$ over $g^{-1}([H])$. More precisely, and taking into account also the singular case, one gets:

\begin{prop}\label{section of E}
For $\mathcal{E}\in \Sigma\simeq M_\cH(3,\mathcal{H},3)$, 
a finite scheme $Z\in S^{[3]}$ belongs to $f^{-1}(\mathcal{E})$ if and only if it is the zero locus of some section of $\mathcal{E}$.

For $\lbrace\cE^\sigma_H,\cE^\tau_H\rbrace$ the fiber over $[H]\notin C$, all subschemes in $f^{-1}(\cE^\sigma_H)$, resp. $f^{-1}(\cE^\sigma_H)$, are a linear section of the pullback via $\psi_H$ of a cubic surface scroll of $\sigma$-type (resp. $\tau$-type). This property characterizes the two fibers.
\end{prop}

\begin{proof}
See \cite[proof of Lemma 8]{kr} for the smooth case. The nodal case is completely analogous, once we have observed that the maximal quadrics $Q_L\subset LG(3,U_6)$, in the singular case of Section \ref{sect cubicscrolls}, are the zero loci of the non-zero sections of the dual tautological vector bundle.
\end{proof}

The moduli space $M_\cH(3,\mathcal{H},3)$ will play a crucial rôle in the sequel. In the next subsection we will take a different perspective on the involution $\varphi$, but the moduli space  will soon return to stage.

\subsection{The nef and movable cones}\label{mov and ample}

The goal of this section is to prove the existence of a (unique) flop over $S^{[3]}$. We will also gather some information about its extremal contractions, that will be made more precise in the next subsections.

Recall that the second cohomology group  
\begin{equation}\label{H2dec}
H^2(S^{[n]},\ZZ)\simeq H^2(S,\ZZ)\oplus \ZZ\delta, 
\end{equation}
where $2\delta$ is the class of the divisor of non-reduced schemes.
The embedding of $H^2(S,\ZZ)$ is given by
$\xi\mapsto \xi_n:=HC^* \xi^{(n)}$, where $HC:S^{[n]}\to S^{(n)}$ is the Hilbert-Chow morphism and $\xi^{(n)}$
is the class in the second cohomology group of  the symmetric product, whose pullback to $S^n$ is the sum of the pullbacks of $\xi$ itself via all the projections $p_i:S^n\to S$, i.e.
$\xi^{(n)}=\sum_{i=1}^n p_i^* \xi$.

\medskip

We first describe the movable cone and the nef cone of $S^{[3]}$. 
The following result is \cite[Proposition 4.1]{beri} ($d=2$ for $t=9$) and \cite[Lemma 3.6]{beri} with 
$n=3$ and $t=4n-3$. We write down a direct proof for the reader's convenience.

\begin{prop}\label{twochambers}\label{invariant is nef}
The nef cone of $S^{[3]}$ is 
$$\mathrm{Nef}(S^{[3]})=\langle \mathcal{H}_3, \mathcal{H}_3-2\delta\rangle .$$
The movable cone has two chambers,  exchanged by the action of $\varphi$:
\[ \mathrm{Mov}(S^{[3]})=\mathrm{Nef}(S^{[3]})\cup \varphi^* (\mathrm{Nef}(S^{[3]})). \]
\end{prop}

\begin{proof}
The structure of the movable cone of $S^{[3]}$ is described in \cite[Proposition 13.1]{bayermacri} in terms of Pell's equations: 
it is the interior of the convex cone generated by $\mathcal{H}_3$ and $17\mathcal{H}_3-9\cdot 4\delta$, the minimal solution of $X^2-18 Y^2=1$ being $(17,4)$. The walls in the movable cone are spanned by vectors of the form $X\mathcal{H}_3-18Y\delta$ for $(X,Y)$ a positive solution of $X^2-72 Y^2=8+\alpha^2$, and $\alpha\in \lbrace 1,2 \rbrace$; this is 
\cite[Theorem 12.1 and Theorem 12.3]{bayermacri}, translated in terms of generalized Pell's equations using \cite[Lemma 2.5]{cattaneo} (see also \cite[Remark 2.8]{cattaneo}). Actually, the vectors cutting chambers of the movable cone are those for which $\frac{2Y}{X}<\frac{4}{17}$. (These Pell's equations will appear again in the proof of Lemma \ref{base irr}, where their relation with the birational geometry of $S^{[3]}$ will become apparent.)

One can immediately check that there are no integral solutions when $\alpha=2$. So we are left to consider solutions of $X^2-72Y^2=9$, which in turn are in the form $(3Z,Y)$ with $Z^2-8Y^2=1$; the latter is a Pell's equation with minimal positive solution $(3,1)$. 
All the solutions of this equation can be found recursively by letting
$$(Z_{k+1},Y_{k+1})=(3Z_k+8Y_k,Z_k+3Y_k).$$ Moreover $\frac{Z_{k+1}}{Y_{k+1}}<\frac{Z_k}{Y_k}$. The first two solutions are $(3,1)$, $(17,6)$. Note that for the second one we have $2\frac{Y_2}{3Z_2}=\frac{4}{17}$, so this solution corresponds to a boundary of the movable cone, and 
there is therefore exactly one wall inside the movable cone. This also means that the ample cone is the interior of the cone generated by $\mathcal{H}_3$ and $9\mathcal{H}_3-18\delta$; this last vector is proportional to $\mathcal{H}_3-2\delta$, which is the generator of the invariant lattice for the action of $\varphi$ by \cite[first case of Proposition 2.2]{beri}. Then the action of $\varphi$ in $\NS(S^{[3]})$ is $-R_{\mathcal{H}_3-2\delta}$, the opposite of the reflection with respect to $\mathcal{H}_3-2\delta$ \textcolor{black}{\cite[Proposition 2.1]{beri}}, exchanging the two rays of the movable cone, hence the two chambers. 
\end{proof}

\medskip\noindent {\it Remark}. 
For future use, let us record that the solution $(X,Y)=(9,1)$ giving the second wall corresponds, going back to \cite[Theorem 12.1]{bayermacri} via \cite[Lemma 2.5]{cattaneo}, to the class $a=-(2,-\cH,5)\in H^*_{alg}(S,\ZZ)$, with $a^2=-2$. Also,  $(a,v)=1$ for $v=(1,0,-2)$. 
(Here and in all the sequel, since $Pic(S)=\ZZ\cH$ we identify $H^*_{alg}(S,\ZZ)$ with $\ZZ\oplus\ZZ\cH\oplus\ZZ$.)
\medskip

The wall-and-chamber decomposition of the movable cone carries a lot of information about the birational geometry of $S^{[3]}$: 
each chamber corresponds to a pair $(M,f)$, where $f:S^{[3]}\dashrightarrow M$ is a birational, not biregular map between hyperK\"ahler manifolds, unique up to automorphisms of the two manifolds \cite[Theorem 1.5 and Lemma 6.22]{markmansurvey}. Walls in the \textit{interior} of the movable cone correspond to flopping contractions and crossing them, in the sense of \cite{bayermacri}, gives rise to flops to hyperK\"ahler birational models.

\begin{coro}\label{unique model}
Any hyperK\"ahler birational model of $S^{[3]}$ is isomorphic to $S^{[3]}$.
\end{coro}

\begin{proof}
Let $f:S^{[3]}\dashrightarrow M$ be a birational map, with $M$ hyperK\"ahler. By Proposition \ref{twochambers}, the pullback of the ample cone of $M$ via $f$ is the interior of either $\rm{Nef}(S^{[3]})$ or $\varphi^*\rm{Nef}(S^{[3]})$. Hence either $f$ or $f\circ \varphi$ sends K\"ahler classes on $S^{[3]}$ to K\"ahler classes on $M$, and then it is an isomorphisms by the Torelli Theorem for hyperK\"ahler manifolds.
\end{proof}

When we fix such an identification $M=S^{[3]}$, $\varphi$ corresponds to the flop associated to the wall and the chamber $\varphi^*(\mathrm{Nef}(S^{[3]}))$ is associated to the pair $(S^{[3]},\varphi)$.
By Bayer-Macrì's study of the Mori cone of $S^{[3]}$ (as a cone in $H^2(S^{[3]},\RR)$, see \cite[Section 12]{bayermacri}), we know that there is an extremal 
contraction associated to $\varphi$.
We call this flopping contraction $c:S^{[3]}\lra N$, with $N$ a normal irreducible projective variety.

The diagram associated to the flop is
\begin{equation}\label{phi as flop}
\xymatrix{ S^{[3]}\ar@{..>}[rr]^{\varphi}\ar[dr]_{c} & & S^{[3]}\ar[dl]^{d}  \\
   & N &}
\end{equation}

As a corollary of Proposition \ref{twochambers}, $\varphi$ is the only flop on $S^{[3]}$! More precisely, it is the only flop for which the space obtained by the transformation is a hyperK\"ahler manifold (contrary to what can happen for certain Mukai flops). This observation will play a key r\^ole in Section \ref{sec: mukai flop}.

\subsection{Interpretations in terms of wall-crossing}
In this section we take a closer look at the description of the flop in terms of moduli spaces of objects in the derived category of $S$. This approach is mainly due to Bayer-Macrì; \cite{bayermacri1,bayermacri,bayer} will be our main references for this section. The main idea is that the moduli space changes when the stability condition does; more precisely, when one crosses certain {\it walls} in the manifold $\Stab(S)$ parametrizing stability conditions. In our setting, the flop $\varphi$ can be realized through such  a wall-crossing, and we will deduce precise information about it.

Actually, the  manifold  $\Stab(S)$ is not completely described, but we will use as a substitute an explicit submanifold $\Stab^\dagger(S)$ constructed in \cite{bridgeland-K3}. 
This submanifold parametrizes   stability conditions of the form $$\sigma_{\alpha,\beta}=
(\Coh^{\beta}(S),
Z_{\alpha,\beta}) \quad \mathrm{for} \quad  
(\beta,\alpha)\in \RR\times \RR_{> 0},$$ 
where $\Coh^{\beta}(S)$ is a certain abelian subcategory of $D^b(S)$, and the central charge is the homomorphism $Z_{\alpha,\beta}=(e^{i\alpha \cH+\beta \cH},-): H^*_{alg}(S,\ZZ)\lra\CC$. In the sequel we stick to these stability conditions and use them to analyze our wall-crossing.

\smallskip
Any non-zero Mukai vector $u$ induces a wall-chamber decomposition of 
$\Stab^\dagger(S)$ \cite[Corollary 3.5]{bayer}, the complement of the walls being the set of stability conditions
$\sigma$ for which any $\sigma$-semistable object with Mukai vector $u$ is also $\sigma$-stable. 

The stability condition $\sigma$ is \textit{generic with respect to $u$} if it does not lie on a wall. In that case, there exists a moduli space $M_\sigma(u)$ of $\sigma$-stable objects with Mukai vector $u$ in the derived category of $S$. For example, for a well-chosen $\sigma_+$, we can get the following moduli spaces:
\begin{itemize}
\item For $a=-(2,-\cH.5)$, $M_{\sigma_+}(-a)$ is a single point \cite[Lemma 7.1]{bayermacri1}, that we can identify with the bundle $\cU_2$ on $S$: by \cite[Theorem 3]{mukai-biregular}, $\cU_2^\vee$ is the unique stable bundle on $S$ whose Mukai vector is $a'=(2,\mathcal{H},5)$ (and the change of sign in the Mukai vector amounts to taking duals \cite[Definition 6.1.2]{huy-lehn}). 
As a consequence, $M_{\sigma_+}(a)$ is also a single point, corresponding to  the shifted bundle $\cU_2[1]$.
   
 \item For $w=(3,-\cH,3)$, $M_{\sigma_+}(w)$ is isomorphic to $\Sigma=M_\mathcal{H}(3,\cH,3)$ (see Theorem \ref{desc Sigma kr}, with a change of sign for taking duals). Indeed, $M_{\sigma_+}(w)$ is a K3 surface by \cite[Lemma 7.2]{bayermacri1}, moreover, there exists a stability condition $\tau$ such that  $M_\tau(w)\simeq \Sigma$ by \cite[Thm 4.4]{bayer}. Since  $M_{\sigma_+}(w)$ is obtained from  $M_\tau(w)$ by some wall-crossings, they must be birational, hence isomorphic. 
\item For $v=(1,0,-2)$, $M_{\sigma_+}(v)$ is isomorphic to $S^{[3]}$, its closed points corresponding to ideal sheaves of length three subschemes of $S$, see \cite[Theorem 4.4]{bayer} and \cite[Example p.71]{sawon-survey}. We will be more precise in Proposition \ref{black box}.
\end{itemize}

Starting from the latter interpretation of $S^{[3]}$ as $M_{\sigma_+}(v)$, 
we will change the stability condition and use wall-crossing to describe the flop $\varphi$. We first locate the relevant wall.

The numerical condition that for some object $T$, its Mukai vector $t$ has the same phase as $v$, that is, $\arg(Z_{\alpha,\beta}(t))=\arg(Z_{\alpha,\beta}(v))$, 
defines either a semicircular or a vertical 
real codimension one submanifold of $\Stab^\dagger(S)$, see \cite[1.3]{maciocia}.
A \textit{potential wall} associated to the primitive lattice containing $\langle v,t\rangle$ is any connected component of such a submanifold \cite[Definition 5.2]{bayermacri}.
Being a potential wall is a necessary condition for being a wall.

Potential walls in $\Stab^\dagger(S)$ that are walls inducing  non-trivial birational transformations
are exactly those for which 
either $t=a$  as in \cite[Theorems 12.1]{bayermacri}, or  $t=s$
 as defined in \cite[Theorems 12.3]{bayermacri}.
Moreover they correspond via wall-crossing to an extremal contraction of $S^{[3]}$ if and only if $\theta_v(t^\perp\cap v^\perp)$ contains the corresponding  wall of the nef cone of $S^{[3]}$, where $\theta_v:v^\perp\lra \NS(S^{[3]})$ is the Mukai morphism \cite[Remark 2.17]{bayermacri}.
Since $S$ has Picard rank one, the latter nef cone has only two walls, giving respectively the Hilbert-Chow contraction for $t=s=(0,0,1)$, and the flopping contraction $c$ for  $t=a=-(2,-\cH,5)$ (for the latter, see the Remark after Proposition \ref{twochambers}).

\medskip\noindent\textit{Definition.} Let $\Lambda$ be the rank two, primitive sublattice 
$$\Lambda =\langle v,a\rangle \subset H^*_{alg}(S,\ZZ), \qquad 
M=\begin{bmatrix}
4 & 1 \\
1 & -2 
\end{bmatrix},$$ 
with $M$  the associated Gram matrix.

\medskip\noindent\textit{Definition.}
Let $W\subset\Stab^\dagger(S)$ be the 
quarter circle parametrized by
\[
\left(\frac{1}{6}cos(t)-\frac{1}{2},\frac{1}{6}sin(t)\right),
 \qquad 0<t<\frac{\pi}{2}.
\]
Let $W'\supset W $ be the semicircle parametrized by $0<t<\pi$. 

\smallskip
By \cite[Theorem 3.1]{bayer} all the pairs $(\beta,\alpha)\in W'$ correspond  to stability conditions, except for $(\beta,\alpha)=(-\frac{1}{2}, \frac{1}{6})$, for which $\mathfrak{R}e( Z_{\alpha,\beta}(-a))=0$.
An explicit computation shows that $Z_{\alpha,\beta}(v) $ and $Z_{\alpha,\beta}(a)$ have the same phase  for any $(\beta,\alpha)\in W'$, hence $W$ is a potential wall for $\Lambda$. By the discussion above, $W$ is in fact a wall,  giving   via wall-crossing the flopping contraction $c$.
A similar computation shows that the walls associated to $\langle 
s,v \rangle$, and giving the Hilbert-Chow contraction, lie in the $\alpha$-axis of the $(\beta,\alpha)$-plane.

\medskip
From now on, we fix a stability condition
$\sigma_0=\sigma_{\alpha_0,\beta_0}$ on the wall $W$; an associated slope function $\nu_0$ can be defined on $\Coh^{\beta_0}(S)$, see \cite[(4)]{bayer}. The full subcategory of $\Coh^{\beta_0}(S)$
of objects with the same phase $\phi=\frac{1}{\pi}\arg(Z_{\alpha_0,\beta_0}(v))$ as $v$ and semistable with respect to $\nu_0$ will be denoted by $\cP(\phi)$.
An object $E$ in the derived category is \textit{$\sigma_0$-semistable of phase $\phi$} if some shift $E[k]$ lies in $\cP(\phi)$, and \textit{$\sigma_0$-stable} if additionally it is simple in $\cP(\phi)$, i.e. it has no non-trivial subobject.
Since $\cP(\phi)$ is abelian of finite length, every object admits a finite Jordan-H\"older filtration, which is trivial if and only if the object is $\sigma_0$-stable. Objects with the same Jordan-H\"older factors are called S-equivalent with respect to $\sigma_0$.

We also fix two stability conditions on the two sides of the wall, namely
$\sigma_\pm=\sigma_{\alpha_0\pm\epsilon,\beta_0}$, which can be chosen to be generic with respect to $v,w$ and $a$
since the set of walls is locally finite, see \cite[Section 6.4]{macrischmidt} (for convenience we keep $\beta_0$ fixed). The following result is due to Bayer and Macrì.

\begin{prop}\label{black box}
There is a natural isomorphism $S^{[3]}\simeq M_{\sigma_+}(v)$ sending $Z$ to
$\cI_Z$. 
Via this isomorphism, the wall $W$ induces via wall-crossing the flopping contraction $c:S^{[3]}\lra N$, more precisely:
\begin{enumerate}
    \item $c$ contracts a curve $C$ if and only if, for $p_1, p_2 \in C$ general, the associated objects in $D^b(S)$ are S-equivalent with respect to $\sigma_0$.
    \item For any $p$ lying in a contracted curve,     
    the Mukai vector of every Jordan-H\"older factor 
    with respect to $\sigma_0$ lies in $\Lambda$. 
\end{enumerate}
Moreover (1), (2) also hold for the contraction $d_{\sigma_0}:M_{\sigma_-}(v)\lra N$ on the other side of the wall.
\end{prop}

\begin{proof}
Stability conditions in $\Stab^\dagger(S)$ corresponding to the natural isomorphism described in the statement correspond to $(\beta,\alpha)$'s with $\alpha\gg 0$ and $\beta<0$, see for example \cite[Theorem 4.4]{bayer}.
The analysis of the potential walls carried before the proof shows that those conditions lie in the same chamber as $\sigma_+$.
Moreover the wall $W$ induces the contraction $c$ since $\RR_{>0}(\cH_3-2\delta)\subset \NS(S^{[3]})$ is a connected component of $\theta_v(v^\perp\cap a^\perp)$.

Now, item (1) is \cite[Theorem 2.18 (b)]{bayermacri}.
For (2), let $Q$ be the last Jordan-H\"older factor of an object $E$. By definition of the filtration, their central charges with respect to $\sigma_0\in W$ are aligned, so their Mukai vectors lie in $\Lambda$ by definition of potential wall associated to $\Lambda$.
Then we can proceed by induction. Finally, the proof is identical for both sides of the wall.
\end{proof}

Different types of walls are defined in \cite[Definition 2.20]{bayermacri}.
We exclude an unpleasant case.

\begin{lemma}\label{not totally semistable}
The wall $W$ is not totally semistable.
\end{lemma}

\begin{proof}
By \cite[Theorem 5.7]{bayermacri}, a wall is totally semistable if there is either an effective spherical class with negative product with $v$, or an isotropic class $u\in \Lambda$ satisfying $(v,u)=1$. 

An isotropic class $u$ with $(v,u)=1$ would generate with $v$ a unimodular, even, rank two sublattice of $\Lambda$, 
which is absurd since $\Lambda$ has discriminant $-9$.
For the first case, we note by solving $(xv+ya)^2=-2$ for $x,y\in \ZZ$, that the only spherical classes in $\Lambda$ are $a$ and $-a$.
By \cite[Proposition 5.5]{bayermacri}, only $a$ is effective, and we are done since $(a,v)=1>0$.
\end{proof}

Since $W$ is not totally semistable, the space $M_{\sigma_0}^s(v)$ of $\sigma_0$-stable objects is non-empty. Then openness of stability \cite[Theorem 4.2]{bayermacri} ensures that $M_{\sigma_0}^s(v)$ admits an open embedding in both $M_{\sigma_+}(v)$ and $M_{\sigma_-}(v)$. The flop 
$$\varphi_{\sigma_0}:S^{[3]}\simeq M_{\sigma_+}(v)\dashrightarrow M_{\sigma_-}(v)\simeq S^{[3]}$$ is induced by this identification, see \cite[proof of Theorem 1.4 (b)]{bayermacri1}; recall that $M_{\sigma_-}(v)\simeq S^{[3]}$ by Corollary \ref{unique model}. This yields an identification of $\varphi_{\sigma_0}$ with $\varphi$ and of $d_{\sigma_0}$ with the flopping contractions $d:S^{[3]}\lra N$, so that we recover \eqref{phi as flop}.

\subsection{The indeterminacy locus}\label{sec: inteterminacy} 

The main advantage of describing the flop in terms of wall-crossing is that  we now have a precise description of the exceptional loci of $c$ and $d=d_{\sigma_0}$.

\medskip\noindent \textit{Definition.} The indeterminacy locus of $\varphi$ will be denoted by $I$. We call \textit{base of the contraction} the image of the exceptional locus of $c$ in $N$.

\medskip

In this subsection we deduce a precise description of $I$, and ultimately show that it identifies with $J$.

\begin{lemma}\label{base irr}
The indeterminacy locus $I$ of $\varphi$ coincides with the exceptional locus of $c$. The base $c(I)$ of the contraction $c$ is an irreducible surface in $N$, birational to $\Sigma$.
\end{lemma}

\begin{proof}
According to Proposition \ref{black box}, to prove the second claim we need to analyse the possible Jordan-H\"older factors of the strictly $\sigma_0$-semistable objects in $D^b(S)$ with Mukai vector $v$. Suppose that, up to a shift, such an object $E$ fits into an exact sequence in $\cP(\phi)$, as an extension of two $\sigma_0$-semistable objects $F_1, F_2$ with Mukai vectors $a_1, a_2=v-a_1\in\Lambda$. For $E$ to be contained  in a curve contracted by $c$, we need $ext^1(F_2,F_1)>1$: by \cite[Theorem 5.7]{bayermacri}, the objects with the same image as $E$ 
are exactly those that are also extensions of $F_1$ by $F_2$.

We claim that the only possible Mukai vectors are $a_1=a, a_2=w$.
Indeed, since $W$ is not totally semistable by Lemma \ref{not totally semistable}, 
any effective class in $\Lambda$ has non-negative intersection with respect to the Mukai pairing with $v$ so, by \cite[Proposition 5.5]{bayermacri},
$(a_1,a_2)$ must be of the form $(a,v-a)$ for $a$ listed in \cite[Remark 2.8]{cattaneo}. This can be translated via \cite[Lemma 2.5]{cattaneo} in terms of solutions to certain Pell's equations, which coincide with the equations that already appeared in the proof of Proposition \ref{twochambers} -- namely the equations $X^2-72Y^2=8+\alpha^2$, such that $X\mathcal{H}_3-18Y\delta$ is a multiple of $\mathcal{H}_3-2\delta$.
We  have shown in the proof of Proposition \ref{twochambers} that the only solution is $(9,1)$, which  corresponds to $a=-(2,-\cH,5)$, and we conclude that the base of $c$ is birational to $M_{\sigma_+}(a)\times M_{\sigma_+}(w)\simeq \{pt\}\times \Sigma$, as claimed.

The same analysis also applies to $d=d_{\sigma_0}$. In fact, \cite[Proposition 9.1]{bayermacri} allows to produce, for a generic element of the base,
a positive dimensional fiber for both $c$ and $d$. In other words, the bases of the contractions $c$ and $d$ coincide. This implies our first claim.
\end{proof}

Now we take a closer look to the subcategory $\Coh^{\beta_0}(S)$ associated to $\sigma_0$ and $\sigma_+$. According to \cite[Section 2]{bayer}, it can be
constructed from two subcategories $T^{\beta_0}$ and $F^{\beta_0}$ of 
$Coh(S)$, as the extension-closure $\langle T^{\beta_0},F^{\beta_0}[1] \rangle$ in $D^b(S)$.
One  gets an abelian category in which short exact sequences are the exact triangles of the derived category whose objects lie in $\Coh^{\beta_0}(S)$. 

An easy computation 
shows that $\cI_Z, \cE^\vee\in T^{\beta_0}$ for any $Z$ in $S^{[3]}$ and any vector bundle $\cE$ corresponding to a point of $\Sigma=M_\cH(3,\cH,3)$. Moreover  $\cU_2\in F^{\beta_0}$.

\begin{lemma}\label{lemma sequence for I}
A length three subscheme $Z$ of $S$ defines a point of $I$ if and only if the sheaf $\cI_Z$
lies in a a short exact sequence
\begin{equation}\label{extension} 
0\lra \cU_2\stackrel{\nu^t}{\lra} \mathcal{E}^\vee\lra \cI_Z\lra 0
\end{equation}
for some vector bundle $\cE$ from $\Sigma=M_\cH(3,\cH,3)$.
Moreover $Z$ and $Z'$ lie in the same fiber of $c$ if and only if the two associated exact sequences \eqref{extension} involve the same bundle $\cE$. 
\end{lemma}

\begin{proof}
By Lemma \ref{base irr}, the only possible Mukai vectors for two objects in $\cP(\phi)$ of which $\cI_Z$ is an extension are $a$ and $w$. An explicit computation of the phase $\phi_+$ of $a,v,w$ with respect to 
$\sigma_+$ shows that $\phi_+(w)<\phi_+(v)<\phi_+(a)$. Therefore the first subobject $A$ in the Jordan-H\"older filtration  of an object $\cI_Z$ in $\cP(\phi)$ has necessarily Mukai vector $w$ \cite[def 2.1 (b)]{bayermacri}.
This yields a short exact sequence in $\cP(\phi)$:
\begin{equation}\label{JH in proof1}
 0\lra A \lra \cI_Z \lra U\lra 0,
\end{equation}
with $v(A)=w$, $v(U)=a$.
Now consider \eqref{JH in proof1} as a sequence in $\Coh^{\beta_0}(S)$.  Objects of $M_{\sigma_+}(v)$ are concentrated in degree $0$ by Proposition \ref{black box}, hence
by \cite[Proposition 2.4]{bayer}, $A$ is a semistable sheaf with Mukai vector $w$, so $A=\cE^\vee$ for some vector bundle $\cE$ corresponding to a point of $\Sigma$. Moreover, again by \cite[Proposition 2.4]{bayer}, \eqref{JH in proof1} is equivalent to a long exact sequence of sheaves
\[
0\lra \cK \lra \cE^\vee \stackrel{\nu^t}{\lra} \cI_Z \lra \cQ\lra 0
\]
with $\cK\in F^{\beta_0}$, of rank two  (we cannot have $\nu^t=0$, since $\cK\in F^{\beta_0}$ and $\cE^\vee\in T^{\beta_0}$). We claim that $\nu^t$ is necessarily surjective. By contradiction, suppose that $\cQ\neq 0$: then $\cQ$ is torsion. Since $\cE^\vee$ is semistable, $\cK$ has slope $\mu_{\cH}(\cK)=\frac{c_1(\cK)\cdot \cH}{2}\leq -\frac{\cH^2}{3}$, so $c_1(\cK)=-\lambda \cH$ with $\lambda>0$.
The Mukai vectors of $\cK$ and $\cQ$ are respectively $v_1=(2,-\lambda \cH,r)$ and $v_2=(0,-(\lambda-1)\cH,r-5)$.
For $v_2$ to be the Mukai vector of a coherent sheaf we need $-(\lambda-1)\geq 0$, so $\lambda=1$. Hence $v_2=(0,0,r-5)$ with $r>5$, which is absurd since then $v_1^2=18-4\cdot r<-2$. So $\cQ=0$, hence $\cK$ has Mukai vector $-a$ and thus $\cK=\cU_2$:
we have obtained (\ref{extension}).
The last assertion is clear.
\end{proof}

\begin{lemma}\label{hom(F,E)}
For any vector bundle $\cE$ from $M_\cH(3,\cH,3)$, we have 
$$hom(\mathcal{E},\cU_2^\vee)=3, \quad ext^1(\mathcal{E},\cU_2^\vee)=ext^2(\mathcal{E},\cU_2^\vee)=hom(\cU_2^\vee,\mathcal{E})=0.$$
\end{lemma}

\proof Riemann-Roch yields $\chi(Hom(\mathcal{E},\cU_2^\vee))=3$. Since $\mathcal{E}$ is semistable, $hom(\cU_2^\vee,\mathcal{E})=0$. Suppose $ext^1(\mathcal{E},\cU_2^\vee)\ne 0$, hence, there is a 
nontrivial extension $0\ra \cU_2^\vee\ra \mathcal{K}\ra \mathcal{E}\ra 0$, with $\mu_\mathcal{H}(\mathcal{K})=2\cH^2/5$. 

We claim that $\mathcal{K}$ must be stable. By contradiction, suppose it contains a destabilizing subsheaf $\mathcal{K}'$, with the induced extension 
$0\ra \mathcal{A}'\ra \mathcal{K}'\ra \mathcal{E}'\ra 0$. If $\mathcal{A}'$ and $\mathcal{E}'$ have rank smaller 
than $\cU_2^\vee$ and $\mathcal{E}$ respectively, then by semistability $c_1(\cA')\le 0$
and $c_1(\mathcal{E}')\le 0$, hence also $c_1(\mathcal{K}')\le 0$, a contradiction.
So $\mathcal{A}'$ or $\mathcal{E}'$ must be zero (otherwise both must have full rank, hence also $\mathcal{K}'$, which is excluded). If $\mathcal{E}'=0$, $\mathcal{K}'$ is a subsheaf of $\cU_2^\vee$ and cannot destabilize $\mathcal{E}$, so neither $\mathcal{K}$. If $\mathcal{A}'=0$, 
then we need $c_1(\mathcal{E}')=\mathcal{H}$, which implies that the inclusion of 
$\mathcal{E}'$ in $\mathcal{E}$ is an isomorphism in codimension one. But then the embedding of $\mathcal{E}'$ into $\mathcal{K}$ can be lifted to $\mathcal{E}$, which is impossible since this would trivialize the extension that defines $\mathcal{K}$. 

Now, the bundle $\mathcal{K}$ being semistable, it defines a point of a moduli space of dimension $v_\mathcal{K}^2+2$, the Mukai vector of $\mathcal{K}$ being $v_\mathcal{K}=(5,2\mathcal{H},8)$. But this gives $v_\mathcal{K}^2+2=
-6$, so the moduli space is empty -- a contradiction!\qed

\begin{lemma}\label{extension2}
Any nonzero $\nu\in Hom(\mathcal{E},\cU_2^\vee)$ yields an extension 
\begin{equation}\label{extension in lemma} 
0\lra \cU_2\stackrel{\nu^t}{\lra} \mathcal{E}^\vee\lra \cI_Z\lra 0
\end{equation}
for some length $3$ subscheme $Z\in S^{[3]}$. 
\end{lemma}

\proof Let $\mathcal{K}$ be the kernel of $\nu^t$ and $\mathcal{Q}$ its image. 
The rank of $\mathcal{K}$ is at most one since $\phi$ is nonzero. 
If $\mathcal{K}$ (hence also $\mathcal{Q}$) has rank one, the stability of $\cU_2$ and $\mathcal{E}$ implies that $c_1(\mathcal{K})\le -\cH$ 
and $c_1(\mathcal{Q})\le -\cH$, hence $c_1(\cU_2)\le -2\cH$, a contradiction. So $\mathcal{K}=0$ since $\cU_2$ is torsion free. 

Note that $\nu$ induces a morphism from $\wedge^2\cU_2=\cO_S(-1)$ to $\wedge^2\cE^\vee=\cE(-1)$, hence a non-trivial section $s$ of $\cE$. Since $\cE$ is stable, $s$ vanishes in codimension two and the associated Koszul complex is exact, hence the sequence $\cU_2\ra\cE^\vee\ra\cO_S$ is also exact.
So the cokernel of $\nu^t$ is a rank one subsheaf of $\cO_S$ with $c_1=0$,
hence the ideal sheaf of some finite scheme $Z$. Since $ch(\cE)=3+\cH$ and $ch(\cU_2)=2+\cH+3[pt]$, we get 
$$ch(\cO_Z)=1-(3-\cH)+(2-\cH+3[pt])=3[pt],$$ 
which means that  $\ell(Z)=3$. 
\qed

\medskip Let $\cF$ be the \textit{twisted} universal rank three vector bundle on 
$S\times\Sigma$ (see \cite[X.2.2 (ii)]{huybrechtsK3} for the definition and existence of $\mathcal{F}$).
Let $p_1$ and $p_2$ be the projections from $S\times \Sigma$ to the two factors. 
Since $\cF$ is twisted with respect to a Brauer class that is pulled-back from $\Sigma$, the push-forward $$\cG= p_{2*}Hom(\cF,p_1^*\cU_2^\vee)$$ makes sense and defines a twisted sheaf on $\Sigma$. By Lemma \ref{hom(F,E)}, $\cG$ 
 is actually a rank three twisted vector bundle on $\Sigma$, and its projectivization is well-defined. Moreover, Lemmas \ref{lemma sequence for I} and \ref{extension2} yield
  a  surjective morphism $$\theta : \PP_\Sigma(\cG)\lra I\subset S^{[3]}.$$

\begin{prop}\label{P=I}
$\theta$ maps $\PP_\Sigma(\mathcal{G})$ isomorphically to $I=J\subset S^{[3]}$ and we get a commutative diagram
\begin{equation}\label{theta}
\xymatrix{
\PP_\Sigma(\mathcal{G})\ar[dr]\ar[rr]^{\theta} && J\ar[dl]^{f}\\
& \Sigma &  }  
\end{equation}
Moreover the fibers of $f$ coincide with those of the restriction of $c$ to $J$.
\end{prop}

\begin{proof}
We consider any $\cI_Z$ lying in an exact sequence as in \eqref{extension in lemma} and we prove that $Z\in J$. To do that, we tensor (\ref{extension in lemma}) by $\cU_2^\vee$ and take cohomology. Since $\cU_2$ is stable on $S$ (hence simple) and rigid, we get an exact sequence 
$$   0\lra End(\cU_2)=\CC\lra Hom(\cE,\cU_2^\vee)\lra H^0(\cI_Z\otimes\cU_2^\vee)\lra 0.$$
Since $Hom(\cE,\cU_2^\vee)$ is three-dimensional by Lemma \ref{hom(F,E)}, we deduce that $H^0(\cI_Z\otimes\cU_2^\vee)$ is two-dimensional, which means that $Z$ only spans a codimension two subspace of $V_7$, that is, $Z\in J$. Moreover, the long exact sequence continues and yields 
$$Ext^1(\mathcal{E},\cU_2^\vee)\lra H^1(\cI_Z\otimes \cU_2^\vee)\lra Ext^2(\cU_2^\vee,\cU_2^\vee)\lra Ext^2(\mathcal{E},\cU_2^\vee).$$
By Serre duality $ext^2(\mathcal{E},\cU_2^\vee)=0$ and $ext^2(\cU_2^\vee,\cU_2^\vee)=1$, while we have $ext^1(\mathcal{E},\cU_2^\vee)=0$ by Lemma \ref{hom(F,E)}. So $h^1(\cI_Z\otimes \cU_2^\vee)=1$, meaning that the extension (\ref{extension in lemma}) is unique.

Hence $\theta$ is injective, 
with image $I=J$; and since  $J$ is smooth by Lemma \ref{cork psi},
$\theta$ must be an isomorphism. Finally, the fibers of $f$ parametrize subschemes of 
$S^{[3]}$ that are S-equivalent, since the S-equivalence class is determined by the isomorphism class of the vector bundle $\cE$ in (\ref{extension in lemma});
so they coincide with the non-trivial fibers of $c$.
\end{proof}

\subsection{The Mukai flop}\label{sec: mukai flop}
Since $\PP_\Sigma(\mathcal{G})\simeq J$ embeds in $S^{[3]}$ as a codimension two subvariety and is a $\PP^2$-fibration, one can perform a Mukai flop 
\begin{equation}\label{MukaiFlop}
\xymatrix{
\ & B\ar[dl]_{b}\ar[dr]^{\check{b}} & \\
S^{[3]}\ar@{.>}[rr]^{\tilde{\varphi}} && M
}
\end{equation}
where $M$ is again simply connected and admits an everywhere non-degenerate two-form \cite[Theorem 0.7]{mukai-flop}.
Recall that $b$ is just the blow-up of $S^{[3]}$ along $J$. 
The exceptional divisor $E$ is a $F\ell_3$-fibration over $\Sigma$;
the contractions to $S^{[3]}$ and to $M$ contract each copy of $F\ell_3\subset\PP^2\times\PP^2$ on one of the two factors.

It is of course tempting to imagine that $\tilde{\varphi}$ should coincide with 
$\varphi$, since as a consequence of Proposition \ref{twochambers} the latter is the unique flop of $S^{[3]}$, and the two transformations  have the same indeterminacy locus $I=J$. But the unicity statement only holds in the algebraic setting, while the result of a Mukai flop needs not always be K\"ahler; a counter-example was given by Yoshioka 
\cite[Section 4.4]{yoshioka}. So we need to provide an extra argument. This will be based on  an explicit extension of $\varphi$ to $B$.

\begin{prop}
The birational involution $\varphi$ lifts to a biregular involution $\tau$ of $B$, preserving the exceptional divisor $E$.
\end{prop}

\begin{proof}
First recall that 
$J$ was defined as the degeneracy locus of the morphism $\phi$ in 
(\ref{morphism-phi}). By Proposition \ref{Zdec}, the rank of $\phi$
never drops to four, so the restriction of $\phi$ to $J$ yields a long exact sequence 
$$0\lra \cQ^\vee\lra V_7^\vee\otimes \cO_J\lra p_*(q^*\cU_2^\vee)_{|J}\lra \cL\lra 0,
$$
for some line bundle $\cL$, and a vector bundle $\cQ^\vee$ whose fiber at $[V_5]$ is the space of linear forms on $V_7$ that vanish on $V_5$. The normal bundle of $J$ in $S^{[3]}$
can thus be described as $Hom(\cQ^\vee,\cL)=\cQ\otimes\cL$; its projectivization,
whose total space is the exceptional divisor $E$, can therefore be identified with $\PP(\cQ)$.
This means that if $Z\in S^{[3]}$ spans $V_5\subset V_7$, a point in $E$
above $Z$ can be identified with a hyperplane $V_6\supset V_5$. 

Recall that $V_5$ is a decomposing five-plane by Proposition \ref{Zdec}, hence $Z$ is contained in a unique hyperplane section $F=X_{ad}(G_2)\cap H$ containing $S$, and $R=F\cap G(2,V_5)$ is a cubic scroll. By Proposition \ref{compatible} and the remark that follows,  $X_{ad}(G_2)\cap G(2,V_6)\cap H$ is a surface $T$ of degree $6$. Thus, $T=R\cup R'$ for another surface $R'$ of degree three. The intersection of $R'$ with $L$ is then a length three subscheme $Z'$ of $S$, and mapping  $(Z,V_6)$ to $Z'$ defines a regular extension $\bar{\varphi}$ of $\varphi$ from  $S^{[3]}$ to $B$, coinciding with $\varphi$ outside $E$.
 
In general $V_6^\perp$ is not isotropic, and $F\ell(V_6)=X_{ad}(G_2)\cap G(2,V_6)$ is a copy of $F\ell_3$ by Proposition \ref{compatible}.
The surface $T$ is therefore cut out in $\PP^2\times \PP^2$ by a pencil of  hyperplanes, defined by $3\times 3$ matrices. Since it contains the cubic scroll $R=\PP^1\times \PP^2\cap H$ 
(see the proof of Proposition \ref{constant rank}), one of the hyperplanes in the pencil has to contain the linear span of $\PP^1\times \PP^2$, 
which implies that it is defined by a matrix of rank one. 
But then the corresponding hyperplane section of $\PP^2\times\PP^2$ is the 
union of $\PP^1\times\PP^2$ with a copy of $\PP^2\times\PP^1$, and $R'$ is a linear section of the latter. So $R'$ is another cubic scroll,
necessarily contained in $G(2,V'_5)$ for some $V'_5\subset V_6$, by \cite[Lemma 7]{kr}. 
This proves that $\bar{\varphi}$ sends $E$ to $J$, and by the universal property of blowups, $\bar{\varphi}$ lifts to an involution of $B$, sending $E$ to itself (and $(Z,V_6)$ to $(Z',V_6)$).
\end{proof}

The restriction of $\tau$ to $E$ is an automorphism over $L^\perp$ via $\pi\circ b_{|E}: E\lra L^\perp$ (see \eqref{J}) i.e. $\pi\circ b_{|E}=\pi\circ b_{|E}\circ \tau$.
Consider the Stein factorization \eqref{stein for J} and let $\lbrace \cE^\sigma_H,\cE^\tau_H\rbrace$ be the fiber over a generic $[H]\in L^\perp$ (recall Theorem \ref{desc Sigma kr} and the Definition that follows).
A general element $Z$ of $f^{-1}(\cE^\sigma_H)$ lies in the pullback via $\psi_H$ of a cubic scroll of $\sigma$-type $R$; for any $(Z,V_6)\in b^{-1}(Z)$, the scroll $R'$ such that $\tau(Z,V_6)=(R'\cap L,V_6)$ is the pullback of a scroll of $\sigma$ or $\tau$-type. By Corollary \ref{section of E}, this does not depend on the choice of $V_6$ and by continuity it also does not depend on the choice of $H$. Thus $\tau$ induces a birational,  hence biregular, involution on $\Sigma$, which is the identity or the covering involution according to whether $R'$ is still of $\sigma$-type or not. We denote this involution by $\gamma$.

\begin{lemma}
The involution $\gamma$ is the covering involution of $g:\Sigma\lra L^\perp$.
\end{lemma}

\begin{proof}
Recall that according to \cite{kr} our scrolls come from $G(3,U_6)$, in which copies of $\PP^2\times \PP^2$ are obtained by decomposing $U_6=A\oplus B$ into the sum of two three-dimensional subspaces, and considering the family of subspaces of the form $a\oplus b^\perp$, for $a,b$ one dimensional in $A$ and $B^\vee$, respectively. When a hyperplane section decomposes into two cubic scrolls, we have seen that it is defined by a rank-one pairing $\nu\in A^\vee\otimes B= \mathrm{Hom}(A,B)$, the condition being that $\langle \nu(a),b\rangle =0$. Then the two components are the families of spaces of the form $a\oplus b^\perp$ for $a\subset \mathrm{Ker}(\nu)$, which contain the  fixed hyperplane $ \mathrm{Ker}(\nu)\oplus B$,   or $b \subset \mathrm{Ker}(\nu^t)$, which contain the  fixed line
$\mathrm{Ker}(\nu^t)^\perp\subset B$. This shows that the two scrolls $R$ and $R'$ have different types. 
\end{proof}

\begin{coro}\label{flopdeMukai}
$\varphi=\tilde{\varphi}$ is a Mukai flop.
\end{coro}

\begin{proof}
Since $\tau$ coincides with $\varphi$ outside $E$, we have $\varphi=b\circ \tau\circ b^{-1}$, so
the birational maps $\varphi$ and $\tilde{\varphi}$ are obtained by composing $b^{-1}$ with two a priori different projections from $B$, respectively $b\circ \tau$ and $\check{b}$.

Consider $p\in \Sigma$. We denote
\[\PP^2_p=c^{-1}(\gamma(p))\quad \quad E_p=b^{-1}(\PP^2_p) \quad \quad \check{\PP}^2_p=\check{b}(E_p)\]
and by $b_p$, $\tau_p$, $\check{b}_p$ the restrictions of the corresponding morphisms.
Note that
$E_p=\lbrace (x,[H])\:|\: x\in H \rbrace\subset \PP^2_p\times \check{\PP}^2_p$ is the flag variety of $\PP^2_p$, and that $b_p$, $\check{b}_p$ are respectively the first and second projection. Also $b_p\circ \tau_p:E_p\lra \check{\PP}^2_p$ is a projection contracting a family of $\PP^1$'s, so it is either $b_p$ or $\check{b}_p$. Suppose that $b_p\circ \tau_p=b_p$ for one $p$. Then $\varphi$ could be extended over $\PP^2_p$ by imposing, for any $Z\in \PP^2_p$, $\varphi (Z)=b\tau(\zeta)$ for any $\zeta\in b^{-1}(Z)$, which is absurd by Proposition \ref{P=I}.
So $b_p\circ \tau_p$ coincide with $\check{b}_p$ for any $p$, hence $(b\circ \tau)|_E:E\lra J$ can be identified with the projection $\check{b}|_E$.
Hence, $\check{b}\circ(b\tau)^{-1}:S^{[3]}\dashrightarrow M$ extends to a biregular map. 
\end{proof}

The following diagram summarizes the situation;  $\tilde{\varphi}$ is the genuine Mukai flop; it is identified with $\varphi$ through the isomorphism between $M$ and $S^{[3]}$, and can be lifted to the regular involution $\tau$ of $B$:

\begin{equation}\label{MukaiFlopagain}
\xymatrix{
\ & B\ar[dl]_{b}\ar[dr]^{\check{b}}\ar[rr]^{\tau} & & B\ar[dr]^{b} & \\
S^{[3]}\ar@{.>}[rr]^{\tilde{\varphi}} && M\ar[rr]^{\simeq} && S^{[3]}
}
\end{equation}

\medskip
As an immediate consequence of 
Corollary \ref{flopdeMukai} together with \eqref{phi as flop}, 
the Mukai flop is a flop also in the sense of \cite[Definition 2.1]{kollar-flops}, with associated contraction $c$.
This implies the smoothness of the base of the contraction, yielding an identification of $f:J\lra \Sigma$ with the restriction of $c$ to $J$, see Proposition \ref{P=I}.

\medskip
Putting together all the results of the last two Sections, we obtain:

\begin{theorem}\label{thm sect 3.6}
The indeterminacy locus of $\varphi$ is the locus $J$ of schemes generating a codimension two linear space in $V_7$.
It is a $\PP^2$-fibration over the K3 surface $\Sigma$ and $\varphi$ is the Mukai flop associated to $J$.
\end{theorem}

\section{The linear system $|\mathcal{H}_3-2\delta|$} 

The extremal, flopping contraction $c:S^{[3]}\lra N$ coincides with the morphism 
$\phi_{|k(\mathcal{H}_3-2\delta)|}$ for $k\gg 0$ (see the proof of \cite[Theorem 8.1.3]{matsuki}). 
In this section we conduct an in-depth study of  the linear system $|\mathcal{H}_3-2\delta|$ itself, and show how it is related to the birational  involution $\varphi$.
We also study the direct sum map $\sigma: S^{[3]}\dashrightarrow \PP(V_7^\vee)$, showing that it is defined by a linear subsystem of $|\mathcal{H}_3-2\delta|$.

\medskip
We start by computing the dimension  of the linear system $|\mathcal{H}_3-2\delta|$. 
The decomposition (\ref{H2dec}) is 
orthogonal with respect to the Beauville-Bogomolov form $q_{BB}$, which restricts to the intersection 
form on $H^2(S,\ZZ)$, hence $q_{BB}(\xi_n)=\xi^2$. On the other hand, $q_{BB}(\delta)=-2(n-1)$ \cite[3.2.1]{debarre}.

\begin{prop}\label{9dim} $|\mathcal{H}_3-2\delta|\simeq \mathbb{P}^9$. 
\end{prop}

\proof 
We know that $\mathcal{H}_3-2\delta$ is nef by Proposition \ref{invariant is nef}. Moreover 
$$q_{BB}(\mathcal{H}_3-2\delta)=\mathcal{H}^2+4\delta^2=18+4\cdot (-4)=2.$$ 
Since the Fujiki constant of $S^{[3]}$ is $15$ according to \cite[3.2.1]{debarre}, this implies that 
 $(\mathcal{H}_3-2\delta)^6=15 q_{BB}(\mathcal{H}_3-2\delta)^3= 120$. In particular, 
the class $\mathcal{H}_3-2\delta$ is big as well. We can therefore invoke Kawamata-Viehweg's theorem and conclude that the 
number of sections is given by the Riemann-Roch polynomial. But on a hyperK\"ahler sixfold $X$ of K3-type we have 
$$\chi(X,\mathcal{L})=\binom{\frac{1}{2}q_{BB}(\mathcal{L})+4}{3}$$
by  \cite[3.3]{debarre},
hence the claim. \qed

\subsection{The secant variety and Pfaffian cubics}
Recall that $S$ is cut in $X_{ad}(G_2)$ by a codimension three linear space $\PP(V_{11})\subset \PP(\fg_2)$.
Let $ |I_3(Sec(S))|$ denotes the linear 
system of cubics of $\PP(V_{11})$ containing the secant variety $Sec(S)$.

\begin{prop}\label{def map}
There is a natural identification 
$$|\mathcal{H}_3-2\delta|\simeq |I_3(Sec(S))|.$$
Moreover, given a length three 
subscheme $Z$ of $S$, not in the base locus of the linear system, its image in $|I_3(Sec(S))|^\vee$ is the 
hyperplane of cubics containing the projective plane spanned by $Z$. 
\end{prop}

\medskip
\noindent\textit{Remark.} 
Since $\mathcal{H}$ is 3-very ample \cite[Theorem 1.1]{knutsen}, any length three subscheme $Z$ of $S$ spans a plane.

\proof First of all, since $\mathcal{H}_3$ is the pull-back, by the Hilbert-Chow morphism, of the line bundle $\mathcal{H}^{(3)}$ on $S^{(3)}$, one has 
$$ H^0(S^{[3]},\mathcal{H}_3)\simeq H^0(S^{(3)},\mathcal{H}^{(3)})\simeq Sym^3H^0(S,\mathcal{H}).$$
Explicitly, given a length three scheme $Z$ 
supported at $p_1+p_2+p_3$, the sum of three points of $S$, corresponding to three lines $\ell_1, \ell_2,
\ell_3\subset \fg_2$, the fiber of $\mathcal{H}^{(3)}$ at $p_1+p_2+p_3$
is $ \ell_1^\vee\otimes \ell_2^\vee\otimes \ell_3^\vee$. Then the cubic polynomial $P$ on $L$ defines a 
section of $\mathcal{H}_3$ whose evaluation at $Z$ is given by the polarization $\tilde{P}$ of $P$, restricted to $ \ell_1\otimes \ell_2\otimes \ell_3$. 

Since $2\delta$ is the class of the divisor parametrizing non-reduced schemes, we need to understand when this section vanishes on the locus of non-reduced schemes. But since the section is pulled-back from $S^{(3)}$, this simply means that $\tilde{P}(x_2,x_2,x_3)=0$ for any  $p_2=[x_2], p_3=[x_3]$ in $S$.
This implies that 
$$\begin{array}{rcl}
P(sx_2+tx_3) & = & s^3\tilde{P}(x_2,x_2,x_2)+3s^2t\tilde{P}(x_2,x_2,x_3)+ \\
 & & \hspace*{1cm} +3st^2\tilde{P}(x_2,x_3,x_3)+t^3\tilde{P}(x_3,x_3,x_3)=0
 \end{array}$$ for any $p_2=[x_2], p_3=[x_3]$ in $S$
and any scalars $s,t$, and the converse is also clearly true. 
This shows that $|\mathcal{H}_3-2\delta|$ identifies with the linear system of cubic polynomials $P$ vanishing on $Sec(S)$.

To prove the second claim, we can suppose that $Z=p_1\cup p_2\cup p_3$ is reduced. The secant variety of $S$ contains the three lines
$\overline{p_1p_2}$, $\overline{p_2p_3}$ and $\overline{p_3p_1}$, so the restriction to the plane
$\langle Z\rangle$ of a cubic polynomial vanishing on $Sec(S)$ is completely determined up to scalar.
If $p_1=[x_1], p_2=[x_2], p_3=[x_3]$, it is clear that $P$ vanishes on the whole plane if and only
if $\tilde{P}(x_1,x_2,x_3)=0$. This concludes the proof. \qed

\begin{coro}\label{factorization}
The map $\phi_{|\mathcal{H}_3-2\delta|}$ associated to  the linear system  $|\mathcal{H}_3-2\delta|$  factorizes through the involution $\varphi$.
\end{coro}

\proof Let $Z=p_1\cup p_2\cup p_3$ be reduced as well as $\varphi(Z)=p_4\cup p_5\cup p_6$. By definition of $\varphi$, for $Z$ general there exists a copy $\fs$ of $\fsl_3$ in $\fg_2$ such that 
the six points $p_1,\ldots , p_6$ belong to $\PP(\fs)$, hence to $\PP(\fs)\cap L$. But the latter is in general
a $\PP^4$, so the two planes $\langle Z\rangle$ and $\langle \varphi(Z)\rangle$  have to meet. Generically, we claim they will 
meet outside the lines of the two triangles defined by $Z$ and $\varphi(Z)$. \textcolor{black}{Indeed, since the planes defined by $p_1,p_2,p_3$ are in general position, those are the only rank two points in  
$\langle Z\rangle$ (where the rank refers to the usual rank 
as skew-symmetric tensors), and the other points on the three sides of the triangle are the only rank four points. Obviously the two planes
cannot meet at a rank two point since otherwise two of the points 
$p_1,\ldots, p_6$ would coincide. They cannot either meet at a rank four point, since then two sides of the two triangles, say $\overline{p_2p_3}$ and $\overline{p_4p_5}$ would meet; but then
$p_2,p_3,p_4,p_5$ would be coplanar, contradicting the fact that 
$p_2,p_3,p_4$ are the only rank two points inside  the plane they
generate. }

\begin{center}
    \begin{tikzpicture}
      \draw [thin] (0,0) -- (2,2) -- (4,1) -- (2,-1) -- (0,0);
      \draw [thin] (0,0) -- (-3.1,1.1) -- (-4,-1) -- (-1,-2) -- (0,0);
       \node[text width=4mm] at (2.5,1.4) {$p_5$};
       \node[text width=4mm] at (2.4,-.25) {$p_6$};
       \node[text width=4mm] at (.8,.3) {$p_4$};
      \draw [ultra thick] (1,0.2) -- (2.6,1.2); 
           \draw [ultra thick] (1,0.5) -- (2.5,0); 
                \draw [ultra thick] (2,-.5) -- (2.2,1.6); 
      \draw [ultra thick] (-.6,-.5) -- (-3,.3); 
           \draw [ultra thick] (-.6,0) -- (-2.6,-1); 
                \draw [ultra thick] (-2,-1.4) -- (-2.8,.7); 
        \node[text width=4mm] at (-3,.5) {$p_1$};
       \node[text width=4mm] at (-.4,-.25) {$p_2$};
       \node[text width=4mm] at (-2.4,-1.25) {$p_3$};
    \end{tikzpicture}
\end{center}

We conclude that a cubic polynomial
 vanishing on 
one triangle plus the span of the other triangle has to vanish on both spans. This proves the claim. \qed

\medskip\noindent {\it Remark}.
Here again there is a strong analogy with the conics on Gushel-Mukai fourfolds
\textcolor{black}{obtained as linear sections}
of copies of $G(2,4)$ inside $G(2,5)$. The linear system of quadratic equations of the Gushel-Mukai, 
which contains the Pl\"ucker quadrics as a hyperplane, restricts to a pencil on the plane  spanned by \textcolor{black}{such a conic}. 
So, exactly as before, containing this plane is just a codimension one condition on the linear system. 
See \cite{im} for more details.

\medskip
Since obviously $Sec(S)\subset Sec(X_{ad}(G_2))$, let us describe the latter. The secant variety is irreducible and reduced, and has dimension 10
according to \cite{kaji}.

\begin{lemma}\label{secsec}
 $Sec(X_{ad}(G_2))=Sec (G(2,V_7))\cap \PP(\fg_2)$. 
\end{lemma}

\proof The inclusion $Sec(X_{ad}(G_2))\subset Sec (G(2,V_7))\cap \PP(\fg_2)$ is obvious, and we only need to prove the reverse inclusion. Note that 
$$Sec (G(2,V_7))=\bigcup_{V_4\subset V_7}\PP(\wedge^2V_4),$$
so it is enough to check that $\PP(\wedge^2V_4\cap\fg_2)\subset 
Sec(X_{ad}(G_2))$ for each $V_4\subset V_7$. Recall that $\fg_2$
is the kernel of the contraction $\wedge^2V_7\lra V_7^\vee$ by the invariant three-form $\omega$. There are two cases. 

If $\omega_{|V_4}=0$, the composition $\wedge^2V_4\lra \wedge^2V_7\lra V_7^\vee$ factors through $V_4^\perp$, hence its kernel has dimension at least three and its intersection with $X_{ad}(G_2)$, being also the intersection with the Grassmannian, is a quadric hypersurface in 
$\PP(\wedge^2V_4\cap\fg_2)$, whose secants span the whole space.

If $\omega_{|V_4}\ne 0$, it is a nonzero decomposable three-form on $V_4$, hence corresponds to a line $V_1\subset V_4$. Then the kernel
of the composition $\wedge^2V_4\lra \wedge^2V_7\lra V_7^\vee\lra V_4^\vee$ is $V_1\wedge V_4$. Since $\PP(V_1\wedge V_4)\subset G(2,V_4)$, we conclude that  $\PP(\wedge^2V_4\cap\fg_2)$ is contained in $X_{ad}(G_2)$, and a fortiori also in its secant variety. \qed 

\medskip 
For a variety like $X_{ad}(G_2)$, which is cut-out by quadrics, the general expectation is that the secant variety 
should be cut out by cubics. Of course, there are many exceptions, typically when the codimension is not large enough. But this is 
the case for $G(2,V_7)$, whose secant variety consists in 
tensors of rank at most four (when 
considered as skew-symmetric forms on $V_7^\vee$), and is cut out 
by the $6\times 6$ Pfaffians. Restricting to $\fg_2$, we call
{\it Pfaffian cubics} those cubics defined, for $v\in V_7$, by
$$ P_v(x) = v\wedge x\wedge x\wedge x \in \wedge^7V_7\simeq\CC, \qquad x\in \fg_2\subset\wedge^2V_7.$$

\begin{prop}\label{pfaff cubics}
There is an exact sequence 
$$0\lra \cO_{\PP(\fg_2)}(-7)\lra V_7^\vee\otimes \cO_{\PP(\fg_2)}(-4)\lra \hspace*{5cm}$$
$$\hspace*{3cm}
\lra V_7\otimes \cO_{\PP(\fg_2)}(-3)\lra \cO_{\PP(\fg_2)}
\lra \cO_{Sec(X_{ad}(G_2))}\lra 0.$$
As a consequence, $I_3(Sec(X_{ad}(G_2)))=V_7$. 
\end{prop}

\proof Recall that $Sec(G(2,V_7))\subset \PP(\wedge^2V_7)$ is 
Cohen-Macaulay of codimension three \cite[Theorem 6.4.1]{weyman}. Lemma \ref{secsec}
shows that $Sec(G(2,V_7))\cap\PP(\fg_2)$ still has codimension three in $\PP(\fg_2)$, and therefore it is also Cohen-Macaulay. It is easy to check that the intersection of $Sec(G(2,V_7))$ with $\PP(\fg_2)$
is transverse at the general point of $Sec(X_{ad}(G_2))$, and since the latter is irreducible and $Sec(G(2,V_7))\cap\PP(\fg_2)$ is Cohen-Macaulay, this intersection must be reduced everywhere. Hence 
$Sec(X_{ad}(G_2))=Sec(G(2,V_7))\cap\PP(\fg_2)$ schematically. Moreover the restriction of the standard locally free resolution 
$$0\lra \cO_{\PP(\wedge^2V_7)}(-7)\lra V_7^\vee\otimes \cO_{\PP(\wedge^2V_7)}(-4)
\lra \hspace*{5cm}$$
$$\hspace*{3cm} \lra V_7\otimes \cO_{\PP(\wedge^2V_7)}(-3)\lra \cO_{\PP(\wedge^2V_7)}
\lra \cO_{Sec(G(2,V_7))}\lra 0$$
to $\PP(\fg_2)$ gives a resolution of the structure sheaf of $Sec(G(2,V_7))$, hence the first claim. We can then twist the resolution by $\cO_{\PP(\fg_2)}(3)$: the first two terms are clearly acyclic, so that taking cohomology we simply get the  exact sequence 
$$0\lra V_7\lra H^0(\cO_{\PP(\fg_2)}(3))\lra H^0(\cO_{Sec(G(2,V_7))}(3))\lra 0.$$ 
This implies the second claim. 
\qed

\subsection{The direct sum map}

Let $Z$ be a length three subscheme of $S$ outside $J$, see Section \ref{sec J}.
If $Z$ is reduced, it corresponds to three planes in general position in $V_7$. 
The hyperplane they generate is simply their direct sum, and we call the corresponding rational map from $S^{[3]}$ to 
$\PP(V_7^\vee)$ the direct sum map. 

\begin{prop}\label{maptohyperplanes}
The direct sum map is a rational map 
\begin{equation}\label{sum_map}
\sigma: S^{[3]}\dashrightarrow \PP(V_7^\vee),
\end{equation}
\textcolor{black}{which is regular outside $J$}, 
and generically finite  of degree $20$.
\end{prop}

\proof For a general hyperplane $V_6$ of $V_7$,  $S\cap G(2,V_6)$  is the union $W$ of six reduced points. These points define six planes in $V_7$ such that 
any three of them generate $V_6$ by Proposition \ref{Zdec}, since a dimension count shows that $V_6$ cannot contain any decomposing five-plane. So any scheme $Z$ in the preimage of $[V_6]$ by $\sigma$ must be the union of three reduced points from $W$, and we conclude that 
$\sigma$ is generically finite of degree $\binom{6}{3}=20$.
\qed

\medskip\noindent {\it \textcolor{black}{Remark.}} The direct sum map $\sigma$ is very similar to the one that, given a smooth quartic surface $T\subset\PP^3$, maps $T^{[2]}$ to $G(2,4)$ 
by sending a length two subscheme of $S$ to the line it generates in $\PP^3$. Obviously, this is a generically finite cover
of degree $6$ (finite if $T$ contains no line).

\medskip

Restricting Pfaffian cubics yields a natural inclusion $j:V_7\hookrightarrow I_3(Sec(S))$, see Proposition \ref{pfaff cubics}, hence a linear projection
$$j^t : \PP(I_3(Sec(S))^\vee)\dashrightarrow \PP(V_7^\vee).$$
This linear subsystem induces a rational map $S^{[3]}\dashrightarrow\PP(V_7^\vee)$. The next easy claim is that this coincides with the direct sum map.

\begin{lemma}\label{commuting triangle}
The following diagram commutes: 
\begin{equation}\label{TwoLinearSystems}
\xymatrix{ &  S^{[3]}\ar@{..>}[dl]_\sigma\ar@{..>}[dr]^{\phi_{|H_3-2\delta|}} & \\  \PP(V_7^\vee) & & \PP(I_3(Sec(S))^\vee).\ar@{..>}[ll]_{j^t}  }
\end{equation}
In particular, $\phi_{|H_3-2\delta|}:S^{[3]}\dashrightarrow \PP(I_3(Sec(S))^\vee)$ is generically finite of degree $d$ such that $d|20$.
\end{lemma}

\begin{proof}
The rational map $j^t$ sends a projective hyperplane $H$ of cubics to the hyperplane in $V_7$ consisting of the Pfaffian cubics $P_v$ that belong to $H$. Consider a general $x=p_1+p_2+p_3\in S^{[3]}$. Choose a basis $e_1,\ldots,e_7$ of $V_7$ 
such that the plane associated to $p_h$ is $\langle e_h,e_{h+3}\rangle$ for $h=1,2,3$. 
Then $x\wedge x\wedge x=6e_1\wedge e_4\wedge e_2\wedge e_5\wedge e_3\wedge e_6$, so $P_v(x)=0$  if and only if 
$v$ belongs to $\langle e_1,\ldots ,e_6\rangle$. This implies the first claim. The second claim follows, since $\sigma$ is generically finite of degree $20$. 
\end{proof}

\subsection{Base point freeness}\label{section bpf}

\begin{theorem}\label{basepointfree}
The linear system $|\mathcal{H}_3-2\delta|$ is base point free. 
\end{theorem}

The proof of Theorem \ref{basepointfree}
that we are going to present, after a few technical preliminaries,  is entirely due to
an anonymous referee. It  will use results  from sections \ref{svb} and \ref{sec J}.

Our initial proof relied on a completely different, and more
involved, deformation argument.
\textcolor{black}{We took a well-chosen $2$-polarized K3 surface $(T,\mathcal{D})$ such that a general deformation, along a divisor in the moduli space, of the pair $(T^{[3]},\mathcal{D}_3)$ is of the form $(S^{[3]},\mathcal{H}_3-2\delta)$ for $(S,\mathcal{H})$ one of our $10$-polarized K3 surfaces. Base point freeness is clear for $|\mathcal{D}_3|$ on $T^{[3]}$, and we deduced base point freeness for $|\mathcal{H}_3-2\delta|$ on $S^{[3]}$, since this is an open condition in a family of polarized manifolds $(X,\mathcal{D})$ such that $h^0(X,\mathcal{D})$ remains constant.}

\smallskip 

Let us denote by $i: J\hookrightarrow S^{[3]}$ the inclusion morphism.
Recall that $J$ was defined in Section \ref{sec J} as the degeneracy locus of the morphism 
$$
\phi: V_7^\vee\otimes \cO_{S^{[3]}}\lra p_*(q^*\cU_2^\vee).
$$

\begin{prop}\label{exactsequence}
There is an exact sequence
$$0\lra p_*(q^*\cU_2^\vee)^\vee\stackrel{\phi^\vee}{\lra} V_7\otimes \cO_{S^{[3]}}\lra
\cO_{S^{[3]}}(\mathcal{H}_3-2\delta)\lra i_*i^*\cO_{S^{[3]}}(\mathcal{H}_3-2\delta)\lra 0.$$
\end{prop}

\proof 
\medskip
Outside $J$, the morphism $\phi$ is surjective. So the kernel of $\phi$ is reflexive of rank one, hence a line bundle. 
Since $J$ has codimension two by Corollary \ref{J irred}, 
we can compute this line bundle by taking determinants. We get
\begin{equation}\label{KerPhi}\mathrm{Ker}(\phi)\simeq \det  p_*(q^*\cU_2^\vee)^\vee\simeq \cO_{S^{[3]}}(2\delta-\mathcal{H}_3)
\end{equation}
(see \cite[Lemma 1.5]{wandel} for the latter isomorphism).

On the other hand, the cokernel of $\phi$ is supported on $J$, and the corank is never bigger than one by Proposition \ref{cork psi}. So $\mathrm{Coker}(\phi)\simeq i_*\cL$ for some line bundle $\cL$ on $J$, and we get an exact sequence
$$0\lra \cO_{S^{[3]}}(2\delta-\mathcal{H}_3)\lra V_7^\vee\otimes \cO_{S^{[3]}}\lra p_*(q^*\cU_2^\vee)
\lra i_*\cL\lra 0.$$
Since $J$ has codimension two in $S^{[3]}$, ${\mathcal Hom}(i_*\cL,\cO) ={\mathcal Ext}^1(i_*\cL,\cO) = 0$. Dualizing the sequence above then yields 
the exact sequence
$$0\lra p_*(q^*\cU_2^\vee)^\vee\lra V_7\otimes \cO_{S^{[3]}}\lra \cO_{S^{[3]}}(\mathcal{H}_3-2\delta)
\lra {\mathcal Ext}^2(i_*\cL,\cO)\lra 0.$$
Finally, the last term is a line bundle on $J$, and therefore the surjectivity of the previous arrow implies that this line bundle is isomorphic to the restriction of
$\cO_{S^{[3]}}(\mathcal{H}_3-2\delta)$. \qed

\medskip
Now, recall from diagram (\ref{RC}) that the family $\cR$ parametrizing cubic scrolls and cones in $X_{ad}(G_2)$ is an incidence correspondence 
\begin{equation}\label{R}
\xymatrix{& \cR\ar[ld]_{\tilde{p}}\ar[rd]^{\tilde{q}} & \\
LieGr(2,V_7) & & X_{ad}(G_2).}
\end{equation}
Moreover, $\cR$ is contained is the family $\cC$ of Schubert cycles 
over $LieGr(2,V_7)$, parametrizing planes in the family $\cV_5$ of decomposing five-planes, that meet their axis from the family $\cA_2$. 
Inside  the incidence correspondence 
\begin{equation}\label{GrLie}
\xymatrix{& Gr_{LieGr(2,V_7)}(2,\cV_5)\ar[ld]_{\hat{p}}\ar[rd]^{\hat{q}} & \\
LieGr(2,V_7) & & G(2,V_7),}
\end{equation}
$\cC$ is the degeneration scheme of the natural morphism
$$\hat{q}^*\cU_2\hookrightarrow \hat{p}^*\cV_5\rightarrow \hat{p}^*(\cV_5/\cA_2). $$
Hence we have the following  Eagon-Northcott resolution 
where $\textcolor{black}{\cH}$ is the pullback of the hyperplane divisor by $\hat{q}$: 
\begin{equation}\label{resC}
0\lra \hat{q}^*\cU_2(-\textcolor{black}{\mathcal{H}})\lra \hat{p}^*(\cV_5/\cU_2)\otimes \cO(-\textcolor{black}{\mathcal{H}})\lra \cO\lra \cO_\cC\lra 0.
\end{equation}
Next, recall from the end of subsection \ref{svb} that $\cR$ is obtained by intersecting $\cC$ with $\PP_{LieGr(2,V_7)}(\cK_5)$  inside $\PP_{LieGr(2,V_7)}(\cA_2\wedge\cV_5)$. Equivalently, it is the 
zero locus of the tautological section of ${\mathcal Hom}(O(-\mathcal{H}),
\hat{p}^*((\cA_2\wedge\cV_5)/\cK_5))$ on $\cC$. Diagram 
\eqref{vb-diagram}
shows that the bundle $ (\cA_2\wedge\cV_5)/\cK_5$ is isomorphic to $\cU_2$, so $\cR$ is defined in $\cC$ as the zero locus of a global section of
 $\hat{p}^*\cU_2\otimes \cO(\textcolor{black}{\mathcal{H}})$;  hence the Koszul resolution
\begin{equation}\label{resR}
0\lra \det (\hat{p}^*\cU_2^\vee)\otimes \cO(-2\textcolor{black}{\mathcal{H}})\lra \hat{p}^*\cU_2^\vee\otimes \cO(-\textcolor{black}{\mathcal{H}})\lra \cO_\cC\lra \cO_\cR\lra 0.
\end{equation}

\begin{lemma}\label{pushforward}
$$\mathbf{R}\tilde{p}_*(\tilde{q}^*\cU_2^\vee(-t\textcolor{black}{\mathcal{H}})_{|\cR})\simeq\left\{
\begin{array}{ll}
    \cV^\vee_5 & \mathrm{if}\; t=0, \\
    0 & \mathrm{if}\; t=1, \\
    \det(\cV_2)^\vee\otimes\det(\cV_5)^{\otimes 2}[-2] & \mathrm{if}\; t=2.
\end{array}
\right. $$
\end{lemma}

As usual the shift notation $[-2]$ means that the corresponding term appears in cohomological degree two, while one gets zero in the other degrees.

\proof First twist (\ref{resC}) by  $\hat{q}^*\cU_2^\vee(-t\textcolor{black}{\mathcal{H}})$ and use the relative 
Borel-Weil-Bott theorem to deduce that 
$$\mathbf{R}\hat{p}_*(\hat{q}^*\cU_2^\vee(-t\textcolor{black}{\mathcal{H}})_{|\cR})\simeq\left\{
\begin{array}{ll}
    \cV^\vee_5 & \mathrm{if}\; t=0, \\
    0 & \mathrm{if}\; 1\le t\le 3, \\
    \det(\cV_5)^{\otimes 2}[-4] & \mathrm{if}\; t=4.
\end{array}
\right. $$
Then use (\ref{resR}) to complete the computation. \qed 

\medskip

Recall that $J$ comes with a natural map $\pi:J\lra \PP^2$, defined in \eqref{J}.

\begin{lemma}\label{isoP2}
One has $i^*\cO_{S^{[3]}}(\mathcal{H}_3-2\delta)\simeq \pi^* \cO_{\PP^2}(1)$.
\end{lemma}

\proof Since the morphism $p: \cZ\lra S^{[3]}$ in \eqref{universal-scheme} is flat, we can use base change to get, from \eqref{KerPhi},
$$i^*\cO_{S^{[3]}}(\mathcal{H}_3-2\delta)\simeq i^*\det  p_*(q^*\cU_2^\vee)\simeq 
\det  p_*(q^*{\cU_2^\vee}_{|\cZ_J}),$$
where $\cZ_J=p^{-1}(J)$. We will use the identification $J\simeq\cD(\psi)$ from Lemma \ref{JisDpsi} to 
compute the right hand side. Indeed, note that the family $\cR$ of cubic scrolls and cones
is contained in $\PP_{LieGr(2,V_7)}(\cK_5)$ and that 
$$\cZ_J\simeq \cR\times_{\PP_{LieGr(2,V_7)}(\cK_5)}\PP_J(\mathrm{Ker}(\psi_{J})).$$
In words, the subscheme $Z\subset S$ of length three corresponding to a point of $J$ is equal to the intersection of the corresponding cubic scroll or cone $\PP(K_5)\cap X_{ad}(G_2)$ with the plane $\PP(\mathrm{Ker}(K_5\lra\fg_2/V_{11}))\subset\PP(K_5)$.
But by definition of $\psi$, we have 
$$\mathrm{Coker}(\psi_{|J})=\pi^* \cO_{\PP^2}(1)\quad \mathrm{and} \quad \mathrm{Ker}(\psi_{J})=\mathrm{Ker}(
\cK_5\lra\pi^*(\Omega(1))).$$ 
Hence the Koszul resolution 
$$0\lra \cO_{\cR_J}(-2\textcolor{black}{\mathcal{H}})\otimes\pi^* \cO_{\PP^2}(1)\lra  \cO_{\cR_J}(-\textcolor{black}{\mathcal{H}})\otimes
\pi^*(\cT(-1))\lra \cO_{\cR_J}\lra \cO_{\cZ_J}\lra 0, $$
where $\cR_{J}=\cR\times_{LieGr(2,V_7)}J$ and $\textcolor{black}{\mathcal{H}}$ is again the pullback of the hyperplane class of
$X_{ad}(G_2)$. Now we can tensorize this exact sequence by $\tilde{q}^*\cU_2^\vee$
and push it forward to $J$. Using Lemma \ref{pushforward}, we get the exact sequence
$$0\lra \cV^\vee_{5|J}\lra 
p_*(q^*{\cU^\vee_2}_{|\cZ_J})\lra 
(\det(\cV_2)^\vee\otimes\det(\cV_5)^{\otimes 2})_{|J}\otimes\pi^*\cO_{\PP^2}(1)\lra 0.$$
Taking determinants, and using (\ref{det}), we deduce the Lemma.\qed

\medskip\noindent {\it Proof of Theorem \ref{basepointfree}.}
Combining Proposition \ref{exactsequence} and Lemma \ref{isoP2}, we get 
the exact sequence 
\begin{equation}\label{sequence-bpf}
0\lra p_*(q^*\cU_2^\vee)^\vee\lra V_7\otimes \cO_{S^{[3]}}\lra \cO_{S^{[3]}}(\mathcal{H}_3-2\delta)
\lra i_*\pi^* \cO_{\PP^2}(1)\lra 0.
\end{equation}
We claim that $|\pi^* \cO_{\PP^2}(1)|\simeq |\cO_{\PP^2}(1)|=\PP^2$. Indeed, the Stein 
factorization of $\pi$ described in Corollary \ref{J irred} yields 
$$\pi_*\cO_J=g_*\cO_{\Sigma}=\cO_{\PP^2}\oplus \cO_{\PP^2}(-3),$$
since $g$ is a degree two cover branched over a sextic curve. We thus get an exact sequence 
\begin{equation}\label{V7'}
    0\lra V_7\lra H^0(\cO_{S^{[3]}}(\mathcal{H}_3-2\delta))\lra H^0(\cO_{\PP^2}(1))\lra 0,
\end{equation}
where a dimension count provides the surjectivity of the rightmost map, see Proposition \ref{9dim}. 
This means that every global section of 
$\pi^* \cO_{\PP^2}(1)$ can be lifted to a global section of $\cO_{S^{[3]}}(\mathcal{H}_3-2\delta)$,
and since the former is base point free, we can conclude from (\ref{sequence-bpf})
that the latter is also base point free. \qed

\medskip
Recall from Lemma \ref{base irr} that the indeterminacy locus $J$ of $\varphi$ is also the exceptional locus of the extremal contraction $c:S^{[3]}\lra N$, induced by a multiple of $\mathcal{H}_3-2\delta$. Since $|\cH_3-2\delta|$ itself is base point free, by \cite[Lemma 2.1.28]{lazarsfeld}, $c$ is also
the first map of the Stein factorization of $\phi_{|\cH_3-2\delta|}$, 
\begin{equation}\label{stein phi}
\xymatrix{ 
   S^{[3]}\ar[rr]^{\phi_{|\cH_3-2\delta|}}\ar[dr]_{c} & &  |\cH_3-2\delta|^\vee \\
 & N\ar[ur]_{\nu}  &  }
\end{equation}

In the next section we will compute the degree of $\phi_{|\cH-2\delta|}$, hence of $\nu$.

In $|\mathcal{H}_3-2\delta|^\vee$ there is a distinguished plane $\PP(V_7^\perp)$, of 
codimension one linear subsystems containing the Pfaffian cubics.
It is the center of the projection $j^t$ from Lemma \ref{commuting triangle}.

\begin{prop}\label{J in P2}
\textcolor{black}{
The image of $J$ via $\phi_{|\mathcal{H}_3-2\delta|}$ is $\PP(V_7^\perp)$.}
\end{prop}

\begin{proof}
A length three scheme $Z$ corresponding to a point in $J$ is contained in $G(2,V_5)$ for some $V_5\subset V_7$. Since  a form in five variables has rank at most four, any Pfaffian cubic 
vanishes on $\langle Z \rangle\subset \PP(\wedge^2 V_5)$.  Hence the claim by Proposition \ref{def map}.
\end{proof}

\noindent {\it Remark}. Here again, the situation is in some sense close to that of Gushel-Mukai varieties in $G(2,V_5)$. The latter Grassmannian is cut out in $\PP(\wedge^2V_5)$
by the Pl\"ucker quadrics,  parametrized by $V_5$ itself. 
A Gushel-Mukai variety is cut out by 
the Pl\"ucker quadrics plus one extra, general quadric. Here, the linear system of cubics 
containing $Sec(S)$ is made of the Pfaffian cubics, plus three extra cubics that remain mysterious. 
Equivalently, the projective plane $\PP(V_7^\perp)$ which is the center of the projection $p_{V_7^\vee}$
remains elusive.

\subsection{Computation of the degree}
Let us finally  determine the degree $d$ of $ \phi_{|\mathcal{H}_3-2\delta|}$. By Corollary \ref{factorization}, 
$d$ is even, and by Lemma \ref{commuting triangle}, $d$ divides $20$. So $d=2,4,10$ or $20$. 
We will exclude the three last cases thanks to  some simple combinatorics.

\begin{prop}\label{degree fiber}
    $d=2$.
\end{prop}

\proof Let us exclude the 
other possibilities. Recall that
for the generic fiber of $\sigma$ there are six distinct points $\lbrace p_i\rbrace_{i=1,\ldots,6}$ such that the fiber is $\{ p_i+p_j+p_k \}_{i<j<k}$. Thus, at the generic point $p_1+p_2+p_3$ of $S^{[3]}$ we have another point $p_4+p_5+p_6$ of $S^{[3]}$ in the same fiber of $\phi_{|H_3-2\delta|}$, and any other point in that fiber must be of the form $p_i+p_j+p_k$ and come with its complement $p_\ell+p_m+p_n$, 
where $\{i,j,k\}\cap \{\ell,m,n\}=\emptyset$. To simplify notations, we will denote these pairs by $(123|456)$ and $(ijk|\ell mn)$, where we can permute
the two triples and the three integers in each triple. 

\smallskip\noindent {\bf $d=4$.} Up to permutation of the indices, we can 
always suppose that the four points in the fiber are given by the 
triples $(123|456)$ and $(124|356)$. Then each triple contains a unique pair of points \textcolor{black}{(either $(12)$ or $(56)$)} that is also contained in another triple. This means that we would be able to split the generic point of $S^{[3]}$ as the sum of a point of $S^{[2]}$ and a point of $S$. But then the universal family $\cZ\lra S^{[3]}$ would admit a rational section, which is impossible since $\cZ$ is irreducible.

\smallskip\noindent {\bf $d=10$.} We will use the same idea as before, although this case is a bit more involved combinatorially. We first remark that up to permuting the indices, there are only three ways to choose $5$ pairs of complementary triples among six indices. We leave the following lemma to the reader. 

\begin{lemma}
 Up to permuting indices, any $5$-tuple of  complementary triples of indices between $1$ and $6$ is equivalent to one of the following:
 $$\begin{array}{ccccc}
 (123|456) &\qquad & (123|456) &\qquad &  (123|456) \\
 (134|256) && (124|356) && (124|356) \\
 (145|236) && (125|346) && (134|256) \\
 (156|234) && (126|345) && (135|246) \\
 (126|345) && (156|234) && (145|236)
 \end{array}$$
\end{lemma}

A nice combinatorial gadget in order to distinguish these $5$-tuples is to associate to them a little graph by using the previous remark that each time 
we have two pairs of complementary triples, we can arrange them in the 
form $(ijk|\ell mn)$ and $(ij\ell | kmn)$. In particular, the pair $(k\ell)$
is distinguished. So each $5$-tuple yields ten such pairs, which we can visualize as the edges of a graph. We get the three following graphs:

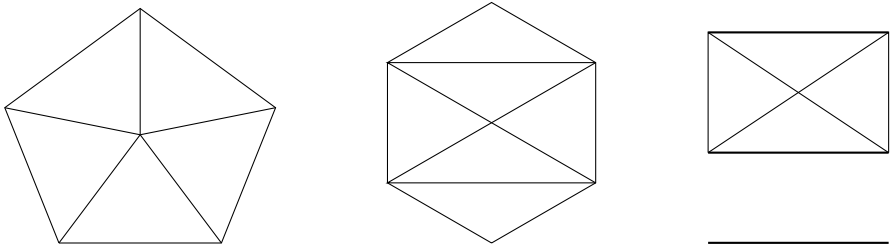
\small
\vspace{-5mm}
\begin{center}
$$\hspace*{-6mm}\begin{array}{ccccc}
    \begin{tikzpicture}[scale = 1.2]
      \draw [thin] (.1,0) -- (1.9,0) -- (2.5,1.5) -- (1,2.6) -- (-.5,1.5) -- (.1,0);
       \draw (1,1.2) node[below]{$1$};
       \draw (.1,0) node[below]{$2$};
       \draw (1,2.6) node[above]{$3$};
       \draw (1.9,0) node[below]{$4$};
       \draw (-.5,1.5) node[left]{$5$};
       \draw (2.5,1.5) node[right]{$6$};
      \draw (1,1.2) -- (1,2.6);
      \draw (1,1.2) -- (2.5,1.5);
      \draw (1,1.2) -- (-.5,1.5);
      \draw (1,1.2) -- (.1,0);
      \draw (1,1.2) -- (1.9,0);
    \end{tikzpicture}
     & \quad &
    \begin{tikzpicture}[scale = .8]
      \draw [thin] (0,0) -- (1.73,1) -- (1.73,3) -- (0,4) -- (-1.73,3) -- (-1.73,1) -- (0,0);
      \draw [thin] (1.73,3)  -- (-1.73,3) -- (1.73,1);
      \draw [thin] (1.73,3)  -- (-1.73,1) -- (1.73,1);  
      \draw  (0,4) node[above]{$1$};
       \draw (0,0) node[below]{$2$};
       \draw (-1.73,3) node[left]{$3$};
       \draw (1.73,3) node[right]{$4$};
       \draw (-1.73,1) node[left]{$5$};
       \draw (1.73,1) node[right]{$6$};
      \end{tikzpicture}
         & \quad & 
    \begin{tikzpicture}[scale = .8]
      \draw [thick] (0,1) -- (3,1);
      \draw [thick] (0,2.5) -- (3,2.5);
      \draw [thick] (0,4.5) -- (3,4.5);
            \draw [thin] (0,2.5)  -- (0,4.5) -- (3,2.5);
      \draw [thin] (3,2.5)  -- (3,4.5) -- (0,2.5); 
      \draw (0,1) node[left]{$1$};
       \draw (0,2.5) node[left]{$2$};
       \draw (0,4.5) node[left]{$3$};
       \draw (3,4.5) node[right]{$4$};
       \draw (3,2.5) node[right]{$5$};
       \draw (3,1) node[right]{$6$};
      \end{tikzpicture}
\end{array}$$
\end{center}
\normalsize

\medskip\noindent
\textcolor{black}{(Beware that on the rightmost graph, the three horizontal edges are double and represented by a thick line).}
In particular the six points do not have the same combinatorial properties, and in each case it is easy to see that in each triple, one point can be distinguished. (In the first graph the vertex $1$ is clearly special, which gives one distinguished point
in five of the ten triples; in the remaining ones like $(456)$, only one vertex, here $5$, is not connected to the other ones and can be distinguished. In the second graph, $1$ and $2$ are special, allowing to distinguish one vertex in each triple that contains one of them or both; in the remaining triples, two vertices are connected to the same special vertex, and the other one is distinguished. 
In the third graph, $1$ and $6$ are special and give a distinguished point in each triple.)
Exactly as in the previous case, we would thus be able to construct a rational section of the universal family $\cZ\lra S^{[3]}$, 
which is impossible.

\smallskip\noindent {\bf $d=20$.} This means that all the triples $p_i+p_j+p_k$ belong to the fiber. 
We denote by $\pi_{ijk}$ the projective plane in $L$ spanned by $p_i,p_j,p_k$.
By Proposition \ref{def map}, we conclude that there is a hyperplane in 
$I_3(Sec(S))$ consisting of cubics that vanish on the $20$ planes 
$\pi_{ijk}$. Recall that the six points $p_1,\ldots ,p_6$ span the
linear space $\PP(\fs)\cap L\subset\PP(\fg_2)$, where $\fs\simeq\fsl_3$, which is a 
$\PP^4$, see the proof of Corollary \ref{factorization}. \textcolor{black}{If the corresponding planes in $V_7$ span the hyperplane $V_6$, this is also }
$\PP(\wedge^2V_6)\cap L\subset \PP(\wedge^2V_7)$. \textcolor{black}{ 
An explicit computation in coordinates shows that} a cubic on this $\PP^4$ vanishing on the $20$ planes must vanish identically. We conclude that the linear system
$|I_3(Sec(S))|$ reduces on $\PP(\fs)\cap L$ to a unique cubic. But since 
the projective four spaces $\PP(\wedge^2V_6)\cap L$ cover \textcolor{black}{an open subset of } 
$L$ \textcolor{black}{(because any element of $\wedge^2V_7$ has rank at most six)}, this would mean that the linear 
system itself reduces to a single cubic, which is absurd. The proof is complete. \qed

\begin{coro}\label{phi is covering inv}
$\varphi$ is the covering involution associated to  $\phi_{|\mathcal{H}_3-2\delta|}$.
\end{coro}

\subsection{Odds and ends}
We conclude the paper with a few observations.

\subsubsection*{Deformations}\label{defo}

The variety  $N$ is normal and $\varphi$ descends to an involution $\bar{\varphi}$ of $N$ which is regular, since it fixes the ample class whose pullback via $c$ is a multiple of $\mathcal{H}_3-2\delta$ \cite[Theorem 2.1.27]{lazarsfeld}. The quotient $N/\langle \bar{\varphi}\rangle$ is still normal,
hence isomorphic to the normalization of $ \phi_{|\mathcal{H}_3-2\delta|}(S^{[3]})$. This is summarized in the 
following diagram:
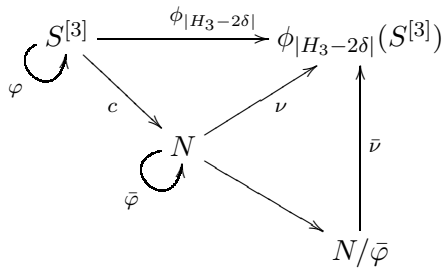
\begin{equation}\label{diagram summarize}
\xymatrix{S^{[3]}\ar[rr]^{ \phi_{|\mathcal{H}_3-2\delta|}}\ar[rd]_c\ar@{-->}@(l,d)_{\varphi} & &  \phi_{|\mathcal{H}_3-2\delta|}(S^{[3]}) \\
 & N\ar[ru]_{\nu}\ar@(l,d)_{\bar{\varphi}}\ar[rd] & \\
 & & N/\bar{\varphi}\ar[uu]_{\bar{\nu}}}
\end{equation}

\medskip\noindent {\it Remark.} Note the analogy with the O'Grady involution, which yields a similar picture where the map $c:S^{[2]}\to N$ contracts a $\PP^2$, $N$ is a double EPW sextic singular in $c(\PP^2)$ which is the inverse image of the Pl\"ucker point, and $N\to \phi_{|\mathcal{H}_3-2\delta|}(S^{[2]})$ is the double cover 
$\tilde{Y}_A\to Y_A$ of a special EPW sextic \cite{ogrady1}. This is a situation where  
$Y_A=\phi_{|\mathcal{H}_3-2\delta|}(S^{[2]})$ is normal (although  $Y_A^{[3]}\neq \emptyset $ does not have 
the expected dimension). 

\smallskip
In our situation, we do not know whether $\bar{\nu}$ is an isomorphism. But pursuing the analogy with 
EPW sextics, it would be tempting to imagine that the double cover $\nu : N\ra \phi_{|\mathcal{H}_3-2\delta|}(S^{[3]})$
can be deformed to a locally complete family of polarized hyperK\"ahler sixfolds. This very natural question is also discussed in \cite{kkm}. 

\subsubsection*{On the contact structure of $X_{ad}(G_2)$}

Since $X_{ad}(G_2)$ is an adjoint variety, it is well known that it admits an algebraic contact structure (and a famous conjecture by  Lebrun and Salamon asserts that all contact structures on Fano manifolds of Picard number one arise from adjoint varieties). Although we did not insist on this, there is an interesting connection between the contact structure on $X_{ad}(G_2)$ and our constructions, which roughly goes as follows. 
 
 Recall from Corollary \ref{decomp=seminull} that $LieGr(2,V_7)\simeq OG(2,V_7)$ parametrizes decomposing five planes 
 $V_5\subset V_7$. In section $2.5$ we have constructed a rank five  subbundle $\cK_5$  of the trivial bundle with fiber $\fg_2$ over $LieGr(2,V_7)$. 
 By  item (2) of Proposition \ref{constant rank}, its fibers are the affine spans of the corresponding cubic scrolls. One can prove that this bundle $\cK_5$ restricted to $X_{ad}(G_2)$ is nothing else than the affine version of the contact vector bundle (sitting inside the affine tangent bundle exactly as the contact hyperplane bundle sits inside the tangent bundle). Moreover, the space of global sections of $\cK_5^\vee$ is isomorphic with $\fg_2$, and a generic section vanishes along two anticanonically embedded Veronese surfaces. So we recover the  two Veronese surfaces of Proposition \ref{two veronese} in a completely different way. 

 Considering the relative zero locus over the projectivized space of sections, and its Stein factorization, we get a double cover of $\PP(\fg_2)$, branched over the sextic hypersurface which is projectively dual to $X_{ad}(G_2)$. This double cover plays a crucial rôle in Homological Projective Duality for the $G_2$-Grassmannian, as developed in \cite{kuznetsov-hpd}.
 It was already observed in loc. cit. that this double cover factorizes the so-called Grothendieck-Springer simultaneous resolution. Ultimately, it is related to the classical fact that a del Pezzo surface of degree six can be represented as the blowup of the projective plane at three points, in exactly two different ways.

    \bibliographystyle{amsplain}

\bibliography{main.bib}

\end{document}